\documentclass[reqno,12pt,centertags]{amsart}
\usepackage[hmargin=2cm,vmargin=2cm]{geometry}
\usepackage{amsmath,amsthm,amscd,amssymb}
\usepackage{upref}
\usepackage{latexsym}
\usepackage{lscape}
\usepackage{graphicx}
\usepackage{tikz}
\usepackage{floatrow}
\usepackage{color}
\usepackage{enumitem}
\usepackage{mathtools}
\usepackage{tikz-cd}
\usepackage{tikz}
\usetikzlibrary{matrix}

\usepackage[abs]{overpic}
\usepackage{xcolor,varwidth}

\makeatletter
\def\theequation{\@arabic\c@equation}

\newcommand{\bbN}{{\mathbb{N}}}
\newcommand{\bbR}{{\mathbb{R}}}

\newcommand{\C}{{\mathbb{C}}}

\newcommand{\bbC}{{\mathbb{C}}}
\newcommand{\bbJ}{{\mathbb{J}}}

\newcommand{\cB}{{\mathcal B}}

\newcommand{\cE}{{\mathcal E}}

\newcommand{\cG}{{\mathcal G}}
\newcommand{\cH}{{\mathcal H}}

\newcommand{\cJ}{{\mathcal J}}

\newcommand{\cL}{{\mathcal L}}

\newcommand{\cO}{{\mathcal O}}
\newcommand{\cP}{{\mathcal P}}

\newcommand{\cV}{{\mathcal V}}

\newcommand{\cX}{{\mathcal X}}

\newcommand{\no}{\nonumber}
\newcommand{\lb}{\label}

\newcommand{\wti}{\widetilde  }

\newcommand{\dL}{\prescript{d\!}{}L}
\newcommand{\dP}{\prescript{d\!}{}P}

\newcommand{\ran}{\text{\rm{ran}}}

\newcommand{\bi}{\bibitem}

\newcommand{\hatt}{\widehat}

\newcommand{\dist}{\operatorname{dist}}

\newcommand{\bfi}{{\bf i}}

\numberwithin{equation}{section}

\newcommand{\dom}{\operatorname{dom}}

\newcommand{\tr}{\operatorname{tr}}

\newcommand{\supp}{\operatorname{supp}}

\renewcommand\Im{\operatorname{Im}}
\renewcommand{\ker}{\operatorname{ker}}

\theoremstyle{plain}
\newtheorem{theorem}{Theorem}[section]
\newtheorem{condition}[theorem]{Condition}
\newtheorem{lemma}[theorem]{Lemma}

\newtheorem{proposition}[theorem]{Proposition}

\theoremstyle{definition}
\newtheorem{definition}[theorem]{Definition}

\newtheorem{example}[theorem]{Example}

\newtheorem{remark}[theorem]{Remark}

\newcommand{\red}{\color{red}}

\DeclareMathOperator{\card}{card}
\DeclareMathOperator{\spec}{Spec}

\begin{document}

\allowdisplaybreaks

\title[Quantum graphs with shrinking edges]{Limits of Quantum Graph Operators With Shrinking Edges}

\author[G.\ Berkolaiko]{Gregory Berkolaiko}
\address{Department of Mathematics,
Texas A\&M University, College Station,
TX 77843, USA}
\email{berko@math.tamu.edu}
\author[Y. Latushkin]{Yuri Latushkin}
\address{Department of Mathematics,
The University of Missouri, Columbia, MO 65211, USA}
\email{latushkiny@missouri.edu}
\author[S. Sukhtaiev]{Selim Sukhtaiev	}
\address{Department of Mathematics,
Rice University, Houston, TX 77005, USA}
\email{sukhtaiev@rice.edu}
\date{\today}

\thanks{Supported by the NSF grants DMS-1067929, DMS-1410657 and
  DMS-1710989, by the Research Board and Research Council of the
  University of Missouri, by SEC Faculty Travel Program and by the
  Simons Foundation. We are grateful to J.~Bolte, P.~Exner and
  U.~Smilansky for the in-depth discussion of the question leading to
  Example~\ref{ex:hyperbolic}.  We thank Th.~Schlumprecht for the
  discussion of Proposition \ref{nw10}.  We are grateful to Y.~Colin
  de Verdi\'ere and S.~Courte for many helpful observations regarding
  the nature of Condition~\ref{newhyp} and definition of $\wti\cL$,
  equation~\eqref{u6}.}


\begin{abstract}
  We address the question of convergence of Schr\"odinger operators on
  metric graphs with general self-adjoint vertex conditions as lengths
  of some of graph's edges shrink to zero.  We determine the
  limiting operator and study convergence in a suitable norm resolvent
  sense.  It is noteworthy that, as edge lengths tend to zero,
  standard Sobolev-type estimates break down, making convergence fail
  for some graphs.  We use a combination of
  functional-analytic bounds on the edges of the graph and Lagrangian
  geometry considerations for the vertex conditions to establish a
  sufficient condition for convergence.  This condition encodes an
  intricate balance between the topology of the graph and its vertex
  data.  In particular, it does not depend on the potential, on the
  differences in the rates of convergence of the shrinking edges, or
  on the lengths of the unaffected edges.
\end{abstract}

\maketitle

{\scriptsize{\tableofcontents}}

\section{Introduction}

Continuous dependence of eigenvalues on edge lengths is a
fundamental issue in the spectral theory of quantum graphs
\cite{BK,M14}.  In particular, it is vital to spectral shape
optimization problems which have received much attention recently (see
for example \cite{F05, EJ, KKM, MR3182688, KKMM, DR, BL,
 MR3688110, MR3611325, Ar} and
references therein).  In such optimization problems achieving extremum
often requires redistribution of volume (edge length) from one edge to
another.  It is thus important to determine the limit of a quantum
graph operator as one or more of the graph's edges shrink to zero.

We answer this question in a very general setting: Schr\"odinger
operators on graphs with general self-adjoint vertex conditions.  The
question naturally breaks into three parts.  First, one has to
determine the domain of the putative limiting operator; this is simple to do on an
intuitive level.  We recall that any set of self-adjoint vertex
conditions is determined by a system of linear relations between the
values of the function $f$ and its derivative $f'$ at the vertices.
Heuristically, the values of the function (and its derivative) at the
end points of an infinitesimally short edge should match.  Hence, it
is natural to conjecture that the vertex conditions for the limiting
operator stem from the augmented linear system
\begin{align}
  \label{eq:satisfy_old}
  &f \text{ satisfies the original vertex conditions and}\\
  \label{eq:equal_on_vanishing}
  &f_e(0)=f_e(\ell_e),\ f_e'(0)=f_e'(\ell_e),
  \quad \text{for every edge $e$ of length $\ell_e\to0$}.
\end{align}  
Eliminating from this system the variables corresponding to the edges of
vanishing length, one obtains the new set of the limiting vertex conditions on the
reduced graph.

The second step is to determine if the vertex conditions obtained
through the above procedure \emph{always define a self-adjoint
  operator} on the new graph.  We answer this question in the positive
by reformulating it in terms of Lagrangian geometry.  It is well known
that self-adjoint extensions of a symmetric operator with equal
deficiency indices are in one-to-one correspondence with the
Lagrangian planes in some symplectic Hilbert space \cite{AS80, BbF95,
  KS99, Ha00, LS1, LSS, McS}.  The question of restricting
self-adjoint vertex conditions from the original graph to its reduced
version --- with some edges shrunk to zero --- is reframed in terms of
the so-called linear symplectic reduction (see, for example,
\cite{McS}) allowing us to show that
(\ref{eq:satisfy_old})-(\ref{eq:equal_on_vanishing}) indeed define a
valid self-adjoint limiting operator.

We now give two simple but illuminating examples of the limiting
vertex conditions.  Consider the graph displayed in the left part of Figure \ref{fg1}. We impose $\delta-$type boundary conditions (cf. \eqref{gh8}) with coupling constants
$\alpha_-$ and $\alpha_+$ at the end points of the
vanishing (middle) edge. Then the limiting vertex condition,  in the right part of Figure \ref{fg1}, is also of
$\delta-$type but with the coupling constant $\alpha_-+\alpha_+$.  An
interesting dichotomy arises when we contract a loop with
$\theta-$periodic conditions (cf. \eqref{gh14}) as shown in Figure \ref{fg4}.  If
$\theta\not=0\ (\text{mod} \ 2\pi)$, shrinking results in two separated
vertices with the Dirichlet conditions (cf. \eqref{gh12}), whereas contracting a periodic
loop (i.e. $\theta=0$) preserves the conditions at the connecting
vertex.  More examples are considered in Section~\ref{sec:main_results}.

The final third step is to investigate convergence of approximating
operators to the limiting operator.  This turns out to be the most
difficult part since the \emph{convergence does not always hold}.  In
Section~\ref{sec:main_results} we construct several examples of
increasing sophistication that illustrate the problem.  Perhaps the
most striking example is that of a sequence of graphs, each with $-1$
as an eigenvalue, whose supposed limit is a positive operator, see
Example~\ref{ex:hyperbolic_fork}.  It turns out to be a delicate job
to craft a condition which excludes all counter-examples and yet
includes all known cases when the convergence does occur.  This is
achieved in Condition~\ref{newhyp} (``Non-resonance Condition'')
which, informally, does not allow eigenfunctions of the approximating
operators to be supported exclusively on the vanishing edges.  We also
show that in some settings which often arise in applications, this
sufficient condition also turns out to be necessary.
Condition~\ref{newhyp} is formulated entirely in terms of the easily
accessible information: the vertex conditions $\cL$ on one hand and
the topological connectivity information from equation
(\ref{eq:equal_on_vanishing}) on the other.  A weaker but more
technical sufficient condition (which follows from Condition 3.2.) is
that the norms of resolvents on the approximating graphs, considered
as operators from $L^2$ to $L^\infty$, remain uniformly bounded as
lengths of some edges shrink to zero.  We point out that such a
boundedness does not hold in general since the standard Sobolev
estimates break down as edge lengths go to zero.

It is important to elaborate on the notion of convergence appropriate for
the operators we consider.  The approximating and the limiting
operators are defined on significantly different spaces making direct
comparison impossible.  Instead we use the notion of \emph{generalized
  norm resolvent convergence}, formulated by O.~Post \cite{P06, P11,
  P12} and P.~Exner \cite{EP} to study the convergence of differential
operators on thin structures to differential operators on graphs.  The
core of the method is to intertwine the spaces of functions supported
on the thick and thin structures by means of quasi-unitary operators.
For illuminating discussion of this subject we refer to \cite[Chapter
4]{P12}.  In our model, the quasi-unitary operators $\cJ_{\ell}$,
formally defined in (\ref{d5}), simply extend by zero the
functions defined on the reduced graph.  This action of the operators
$\cJ_{\ell}$ and their quasi-inverses $\cJ_{\ell}^*$ is schematically
illustrated in Figure \ref{fg1}.

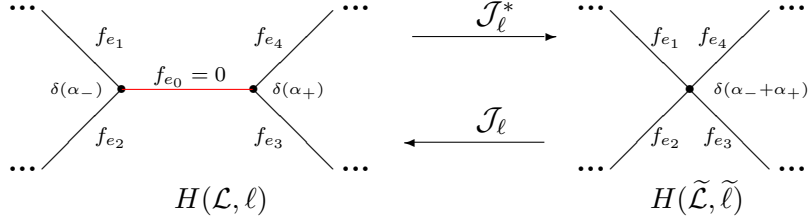
\begin{figure}
	\begin{floatrow}		
		\begin{picture}(0,100)(0,0)
		\put(-120,50){\circle*{3}}
		\put(-70,50){\circle*{3}}
		\put(-108,53){\tiny $f_{e_0}=0$}  
		\put(-70,68){\tiny $f_{e_4}$}  
		\put(-70,30){\tiny $f_{e_3}$}  
		\put(-130,68){\tiny $f_{e_1}$}  
		\put(-130,30){\tiny $f_{e_2}$}  
		\put(-63,50){\tiny $_{\delta(\alpha_+)}$} 
		\put(-147,50){\tiny $_{\delta(\alpha_-)}$} 
		\put(-37,80){{\bf...}}
		\put(-37,19){{\bf...}}
		\put(-163,19){{\bf...}}
		\put(-163,80){{\bf...}} 	
		\thinlines	
		{\red\put(-120,50){\line(1,0){50}}}
		\put(41.5, 70){\vector(1,0){3}}
		\put(-10,70){\line(1,0){50}}
		\put(10,75){ $\cJ^*_{\ell}$} 
		
		\put(-150,20){\line(1,1){30}}
		\put(-150,80){\line(1,-1){30}}
		\put(-70,50){\line(1,1){30}}
		\put(-70,50){\line(1,-1){30}}     
		\put(-100,5){\small $H(\cL, \ell)$}  
		
		\thinlines
		\put(-10,29.8){\vector(-1,0){3}}
		\put(-10,30){\line(1,0){50}}
		\put(10,35){ $\cJ_{\ell}$} 
		
		\put(98,68){\tiny $f_{e_4}$}  
		\put(100,30){\tiny $f_{e_3}$}  
		\put(80,68){\tiny $f_{e_1}$}  
		\put(80,30){\tiny $f_{e_2}$} 
		\put(95,50.2){\circle*{3}}
		\put(128,80){{\bf...}}
		\put(128,19){{\bf...}}
		\put(52,19){{\bf...}}
		\put(52,80){{\bf...}}
		\put(104,50){\tiny $_{\delta(\alpha_-+\alpha_+)}$} 
		
		\put(65,20){\line(1,1){30}}
		\put(65,80){\line(1,-1){30}}
		\put(95,50){\line(1,1){30}}
		\put(95,50){\line(1,-1){30}}   	
		\put(80,5){\small $H(\wti\cL, \wti \ell)$}   
		
		\end{picture}
		\caption{A vanishing edge (horizontal) $e_0$
                  connecting two vertices equipped with the
                  $\delta$-type boundary conditions. Quasi-unitary
                  operators map the spaces of functions supported on
                  respective graphs and ``almost'' intertwines
                  the operators on the corresponding graphs.}
                \label{fg1}
	\end{floatrow}
\end{figure} 

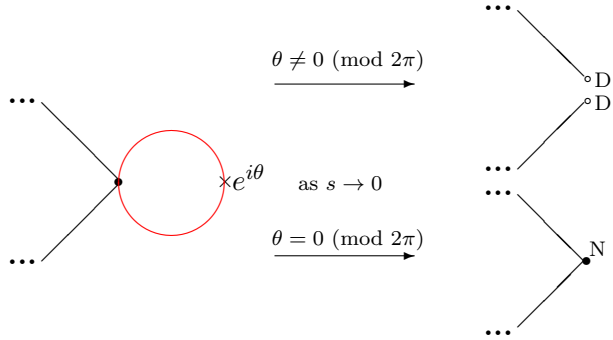
\begin{figure}[h]
	\begin{floatrow}		
		\hspace{-2.5cm}\begin{picture}(100,130)(0,0)
		\put(-11,55){\circle*{3}}
		\put(26,53){\bf \tiny$\times$}
		\put(32,52){\small $e^{i\theta}$}  
		\put(-53,24){{\bf...}}
		\put(-53,85){{\bf...}} 	
		\put(-40,25){\line(1,1){30}}
		\put(-40,85){\line(1,-1){30}}  
		{\red\put(9,54.5){\circle{41}}}

		\put(48,92){\line(1,0){50}}  
		\put(47,99){\tiny{$\theta\not=0\ (\text{mod}\  2\pi)$}}  
		\put(98,92){\vector(1,0){3}}
		\put(57,52){\tiny{as $s\rightarrow 0$}}  
		
		\put(48,27){\line(1,0){50}}  
		\put(47,31){\tiny{$\theta=0\ (\text{mod}\  2\pi)$}}  
		\put(98,27){\vector(1,0){3}}
		
		\put(127,59){{\bf...}}
		\put(127,120){{\bf...}}
		\put(166.5,94){\circle{2}}
		\put(166.5,85.5){\circle{2}}	
		\put(169,82){\text{\tiny D}}	
		\put(169,91){\text{\tiny D}}
		\put(140,60){\line(1,1){25}}
		\put(140,120){\line(1,-1){25}}

		\put(127,-3){{\bf...}}
		\put(127,50){{\bf...}}
		\put(166,25){\circle*{3}}
		\put(167,27){\tiny N}
		\put(140,0){\line(1,1){25}}
		\put(140,50){\line(1,-1){25}}
		\end{picture}
		\caption{Loop of length $s$.}\lb{fg4}
	\end{floatrow}
\end{figure}

We now summarize previous related work. In \cite{BK12} it was shown
that the eigenvalues of the Schr\"odinger operator with arbitrary
vertex conditions depend analytically on the edge lengths, as long as
they remain strictly positive.  On the opposite side of the spectrum
are the results of \cite{HS}, where the behavior of the eigenvalues of
the Schr\"odinger operators with matrix valued potential on $[0,s]$
was studied as $s\rightarrow 0$.  The case of diagonal potential here
is equivalent to a bipartite graph with all edges of the same length.
Band and Levy \cite[App.~A]{BL} gave an informal argument for
eigenvalue convergence for the case of shrinking to zero edges that
link vertices with Neumann--Kirchhoff, i.e. $\delta$-type with zero
coupling constant, conditions.  They approached the problem via a
secular determinant which is only viable for scale invariant boundary
conditions and zero potential.  Finally, perhaps the most directly
related reference is Cheon, Exner and Turek \cite{CEO}, which resolves
a longstanding open problem about approximating a vertex with
arbitrary conditions by a graph with internal structure but only
$\delta$-type conditions.  As the approximating graph is shrunk to a
point, the authors allow $\delta$-couplings to vary, calculate
Green's function explicitly and thus establish convergence.  We note
that the case of a graph with fixed $\delta$-type conditions is
covered by our results via Lemma~\ref{1a}. Finally we remark that our
methods are not incremental extensions of the above mentioned works
but a new combination of functional-analytic estimates and Lagrangian
geometry considerations.

This paper is organized as follows. In Section \ref{sec2} we discuss a
one-to-one correspondence between Lagrangian planes and self-adjoint
boundary conditions on metric graphs.  Section \ref{sec:main_results}
summarizes main results of this paper illustrated by numerous
examples.  Section~\ref{sec5} reviews relevant definitions and results
from linear symplectic geometry, proves self-adjointness of the
limiting operator, and explores the geometrical meaning of
Condition~\ref{newhyp}.  Functional-analytic estimates producing the
main result are presented in Section~\ref{aux}.

{\bf Notation.} We denote by $I_n$ the $n\times n$ identity matrix.
For an $n\times m$ matrix $A=(a_{ij})_{i=1,j=1}^{n,m}$ and a
$k\times\ell$ matrix $B=(b_{ij})_{i=1,j=1}^{k,\ell}$, we denote by
$A\otimes B$ the Kronecker product, that is, the $nk\times m\ell$
matrix composed of $k\times\ell$ blocks $a_{ij}B$, $i=1,\dots n$,
$j=1,\dots m$. We let $\langle\cdot\,,\cdot\rangle_{\C^n}$ denote the
complex scalar product in the space $\C^n$ of $n\times 1$ vectors.
We denote by $\cB(\cX)$ the set of linear bounded operators and by
$\spec(T)$ the spectrum of an operator $T$ on a Hilbert space
$\cX$. Given a subspace $S\subset \cX$ we denote
$\prescript{d}{}S:= S\oplus S$.  Given an operator $T$ acting in $\cX$
we denote $\prescript{d}{}T:=T\oplus T$, then $\prescript{d}{}T$ acts
in $\prescript{d}{}\cX$. Given two subspace $U,V\subset \cX$, we write
$\cX:=U\dot{+}V$ if $U\cap V=\{0\}$ and $U{+}V=\cX$.

We denote $\bbR_{>0}^d:= (0,\infty)^d,\ \bbR_{\geq0}^d:= [0,\infty)^d$, $d\in\bbN.$ Given two positive quantities $x, y$ we write $x\lesssim_{\alpha}y$ if there exits a positive constant  $c=c(\alpha)>0$ depending only on $\alpha$ such that $x\leq c(\alpha) y$, likewise $x\lesssim y$ if and only if $x\leq C y$ for some absolute constant $C>0$. Given an edge $e$ incident to a vertex $v$ we write $e \sim v$.

\section{Preliminaries and Notation}\label{sec2}
\subsection{Schr\"odinger Operators on Graphs With Fixed Edge Lengths.} We begin by discussing  differential operators on metric graphs. To set the stage, let us fix a discrete graph $\cG=(\cV,\cE)$ where $\cV$ and $\cE$ denote the set of vertices and edges correspondingly. We assume that $\cG$ consists of finite number $|\cV|$ of vertices and finite number  $|\cE|$ of edges. Each edge $e\in\cE$  is assigned positive length $\ell_{e}\in(0,\infty)$ and some direction. The corresponding metric graph is denoted by $\Gamma$. The boundary $\partial\Gamma$ of the metric graph is defined as follows,
\begin{equation}
\partial\Gamma:=\cup_{e\in\cE} \{a_e,b_e\}, 	
\end{equation}
where $a_e, b_e$ denote the end points of edge $e$.  Then, one has
\begin{equation}\label{vv1}
L^2(\partial\Gamma)\cong \bbC^{2|\cE|},
\end{equation}
where the space $L^2(\partial\Gamma)=\bigoplus_{e\in\cE}\left( L^2(\{a_e\})\oplus L^2(\{b_e\})\right)$  corresponds to the discrete Dirac measure with support $\cup_{e\in \cE} \{a_e, b_e\}$.
Let us introduce the following spaces of functions
\begin{align}
&L^2(\Gamma):=\bigoplus_{e\in\cE}L^2(e),\  \hatt{H}^k(\Gamma):=\bigoplus_{e\in\cE}H^k(e),\ k\in\bbN\no,
\end{align}  
where $H^k(e)$ is the standard $L^2$ based Sobolev space of order $k\in \bbN$.  
The Dirichlet and Neumann trace operators are defined by the formulas
\begin{align}
&\gamma_D: \hatt{H}^2(\Gamma)\rightarrow L^2(\partial \Gamma),
\ \gamma_Df:=f|_{\partial \Gamma}, f\in \hatt{H}^2(\Gamma),\label{vv2}\\
&\gamma_N: \hatt{H}^2(\Gamma)\rightarrow L^2(\partial \Gamma),
\ \gamma_N f:=\partial_{\nu}f|_{\partial \Gamma}, f\in \hatt{H}^2(\Gamma),\label{vv3}
\end{align}
where $\partial_{\nu} f$ denotes the inward derivative of $f$.
The trace operator is a bounded, linear operator given by
\begin{equation}\lb{2.4new}
\tr:=
\left[\begin{matrix}
\gamma_D\\
\gamma_N
\end{matrix}\right],\, \tr: \hatt{H}^2(\Gamma)\rightarrow L^2(\partial \Gamma)\oplus L^2(\partial \Gamma)\cong \bbC^{4|\cE|}.
\end{equation}
This notation gives rise to the following form of the second Green's identity,
\begin{equation}\label{vv5}
  \int_{\Gamma} \overline{f''}g-\overline{f}g''
  =-\int_{\partial\Gamma}
  \overline{\partial_{\nu}f}g-\overline{f}\partial_{\nu}g 
  = \langle\tr f, [J\otimes I_{2|\cE|}]\tr g\rangle_{\bbC^{4|\cE|}}.
\end{equation}
Finally, the Sobolev space of functions vanishing on the boundary $\partial\Gamma$ together with their derivatives is denoted by
\begin{equation}\label{b4}
\hatt{H}^2_0(\Gamma):=\left\{f\in \hatt{H}^2(\Gamma): \tr f=0\right\}.
\end{equation}

Next, we introduce the minimal Schr\"odinger operator $H_{min}$ and its adjoint $H_{max}$. To this end, let us fix a bounded real-valued potential $q\in L^{\infty}(\Gamma;\bbR)$. Then the linear operator 
\begin{equation}\label{b5}
H_{min}:=-\frac{d^2}{dx^2}+q,\quad \dom(H_{min})=\hatt H^2_0(\Gamma),
\end{equation}
is symmetric in $L^2(\Gamma)$. Its adjoint $H_{max}:=H_{min}^*$ is given by the formulas 
\begin{equation}\label{b6}
H_{max}:=-\frac{d^2}{dx^2}+q,\quad \dom(H_{max})=\hatt{H}^2(\Gamma).
\end{equation}
Moreover, the deficiency indices of $H_{min}$ are finite and equal, that is,
\begin{equation}
0<\dim\ker(H_{max}-\bfi)=\dim\ker(H_{max}+\bfi)<\infty.
\end{equation}
By the standard von-Neumann theory, the self-adjoint extensions of
$H_{min}$ exist and every self-adjoint extension $H$ satisfies
$H_{min}\subset H=H^*\subset H_{max}$. There are various possible
parameterizations of all self-adjoint extensions of the minimal
operator. In this paper we utilize the one stemming from symplectic
geometry. Namely, we use the fact that the self-adjoint extensions of
the minimal operator are in one-to-one correspondence with the
Lagrangian planes in some symplectic Hilbert space \cite{AS80, McS,
  Pa}. This relation was noted by many authors in different forms,
cf., e.g, \cite{BbF95, Ha00, KS99,
  LS1}. 
For the sake of completeness we provide its proof in
Section~\ref{sec5} after recalling the definition of Lagrangian subspaces
of a symplectic space.

\begin{proposition}[cf.\ \cite{Ha00, KS99, KS06}]\label{b1}
  Assume that $q\in L^{\infty}(\Gamma;\bbR)$. Then the self-adjoint
  extensions of $H_{min}$ $($cf.\ \eqref{b5}$)$ are in one-to-one
  correspondence with the Lagrangian planes in
  $\dL^2(\partial \Gamma)$ equipped with the symplectic form $\omega$
  given by
  \begin{align} 
    \label{vv9}
    &\omega:\  \dL^2(\partial \Gamma)\times\,\dL^2(\partial
      \Gamma)\rightarrow \bbC, \\
    \label{eq:def_omega}
    &\omega( (\phi_1, \phi_2), (\psi_1, \psi_2)):=\int_{\partial\Gamma} \overline{\phi_2}\psi_1-\overline{\phi_1}\psi_2,\\
    & (\phi_1, \phi_2), (\psi_1, \psi_2)\in \dL^2(\partial\Gamma).\label{vv10}
  \end{align}
  Namely, the following two assertions hold.
  
  1) If $H$ is a self-adjoint extension of $H_{min}$ then 
  \begin{equation}\no
    \cL({H}):=\tr\big({\dom(H)}\big) \text{\   is a Lagrangian plane in\ }\dL^2(\partial \Gamma).
  \end{equation}
  Moreover, the mapping $H\mapsto \cL({H})$ is injective. 

  2) Conversely, if $\cL\subset \dL^2(\partial \Gamma)$ is a Lagrangian plane then the operator 
  \begin{align}\label{b7}
    H({\cL}):=-\frac{d^2}{dx^2}+q(x),\ \dom\big(H({\cL})\big)=\{f\in \hatt{H}^2(\Gamma): \tr f\in\cL\},
  \end{align}
  is a self-adjoint extension of $H_{min}$. 
\end{proposition}

We recall a related description of the domain of $H(\cL)$: There exist
three orthogonal projections $P_D, P_N, P_R$ acting in
$L^2(\partial\Gamma)$, referred to as the Dirichlet, Neumann, and
Robin projections respectively, such that
\begin{equation}\label{new4}
  L^2(\partial\Gamma)=\ran(P_D)\oplus\ran(P_N)\oplus\ran(P_R),
\end{equation}
and an invertible, self-adjoint matrix $Q$ such that
\begin{equation}\label{new5}
  \dom(H(\cL))=\left\{f\in\hatt H^2(\Gamma)\Big|\begin{matrix}
      P_D \gamma_D f=0,\,P_N \gamma_Nf=0,\,\\
      P_R \gamma_Nf= Q P_R\gamma_Df
    \end{matrix}\right\},
\end{equation}
cf., e.g., \cite[Theorem 1.1.4]{BK}.  In this notation for arbitrary
$f\in\dom(H(\cL, \ell))$ one has
\begin{equation}\label{aa14}
    \langle f, H(\cL,\ell)f\rangle_{L^2(\Gamma(\ell))}
    = \|f'\|^2_{L^2(\Gamma(\ell))} 
    + \langle f, q^{\ell}f\rangle_{L^2(\Gamma(\ell))}
    + \langle P_R \gamma_D^{\ell} f, Q P_R \gamma_D^{\ell} f\rangle_{L^2(\partial\Gamma)}.
\end{equation}

The vertex conditions are called {\it scale invariant} if $P_R=0$,
cf. \cite[Section 1.4.2]{BK}.  Conditions are scale invariant if and
only if the corresponding Lagrangian plane $\cL \subset L^2(\partial
\Gamma)\oplus L^2(\partial \Gamma)$ decomposes as $\cL = \cL_D \oplus
\cL_N$, see Proposition~\ref{gh1} in Section~\ref{sec5}.

Next, we list some standard conditions at a vertex $v$ (here
$\partial_{\nu} f$ denotes the inward derivative of $f$):
\begin{itemize}
\item  $\delta$-type condition with coupling constant $\alpha\in\bbR$:
  \begin{equation}
    \label{gh8}
    \begin{cases}
      f\text{\ is continuous at\ }v,\\
      \sum\limits_{v\sim e}\partial_{\nu}f(v)=\alpha f(v),
    \end{cases}
  \end{equation}
\item Neumann--Kirchhoff condition is given by \eqref{gh8} with
  $\alpha=0$,
  \begin{equation}
    \label{gh9}
    \begin{cases}
      f\text{\ is continuous at\ }v,\\
      \sum\limits_{v\sim e}\partial_{\nu}f(v)=0,
    \end{cases}
  \end{equation}
\item $\delta'$-type condition with coupling constant $\alpha\in\bbR$:
  \begin{equation}
    \label{gh10}
    \begin{cases}
      \partial_{\nu}f\text{\ is continuous at\ }v,\\
      \sum\limits_{v\sim e}f(v)=\alpha \partial_{\nu}f(v),
    \end{cases}
  \end{equation}
\item anti-Kirchhoff condition is given by \eqref{gh10} with
  $\alpha=0$, 
  \begin{equation}
    \label{gh11}
    \begin{cases}
      \partial_{\nu}f\text{\ is continuous at\ }v,\\
      \sum\limits_{v\sim e}f(v) = 0,
    \end{cases}
  \end{equation}
\item Dirichlet conditions
  \begin{equation}
    \label{gh12}
    f_e(v)=0, \qquad \mbox{for all }e\sim v,
  \end{equation}
\item $\theta$-periodic (magnetic) condition at a vertex of degree 2
  with incident edges $e_1$ and $e_2$ is given by
  \begin{equation}
    \label{gh14}
    \begin{cases}
      f_{e_1}(v) = e^{i\theta} f_{e_2}(v), \qquad \theta\in\bbR, \\
      \partial_\nu f_{e_1}(v) = -e^{i\theta} \partial_\nu f_{e_2}(v).
    \end{cases}
  \end{equation}
\end{itemize}

\subsection{Schr\"odinger Operators on Graphs With Vanishing Edges.} 

The main purpose of this paper is to investigate convergence of the
spectral projections of the Schr\"odinger operators on $\Gamma(\ell)$,
where $\ell=(\ell_e)_{e\in\cE}$ denotes the vector of edge lengths, as
\begin{equation}
  \label{eq:l_ltilde}
   \ell\rightarrow\wti \ell \text{\ in\ }\bbR^{|\cE|},
   \qquad \mbox{where}\quad
   \ell\in\bbR^{|\cE|}_{>0}
   \quad\mbox{and}\quad
   \wti \ell=(\wti\ell_e)_{e\in\cE}\in\bbR_{\geq 0}^{|\cE|}\setminus\{0\}.
\end{equation}
Note that the components of $\ell$ are all positive, whereas some, but
not all, components of $\wti \ell$ are equal to zero.  The
``limiting'' metric graph $\Gamma(\wti\ell)$ is based on the discrete
graph $\wti\cG$ obtained from $\cG$ by contracting the edges with
$\wti\ell_e=0$.

We emphasize that the main difficulty is in dealing with the edges
whose lengths tend to zero.  For notational convenience we label edges
of the graph $\cG$ so that the first $m$ ones are rescaled but not
completely shrunk to zero, and the remaining $|\cE|-m$ edges are being
shrunk to zero as $\ell\rightarrow\wti\ell$, that is, we write
\begin{equation}
\wti \ell=(\wti \ell_{e_1}, ..., \wti \ell_{e_m}, 0, ... 0)^{\top}, \label{3.1}
\end{equation}
where the first $m\geq1$ components of $\wti\ell$ are positive.
To simplify notation we denote the set of the {\it non-vanishing} edges of $\Gamma( \ell)$ by
\begin{equation}
\cE_{+}:=\{e_1,...,e_m\},
\end{equation}
and the {\it vanishing} ones by
\begin{equation}
\cE_{0}:=\{e_{m+1},...,e_{|\cE|}\}.\label{3.3}
\end{equation}
Let $\Gamma_{+}(\ell)$ be the subgraph of $\Gamma(\ell)$ with the set of edges $\cE_+$, and let $\Gamma_{0}(\ell)$ be the subgraph of $\Gamma(\ell)$ with the set of edges $\cE_0$. In particular, one has
\begin{equation}
\Gamma(\ell)=\Gamma_{+}(\ell)\cup\Gamma_{0}(\ell).\label{3.4}
\end{equation}
Let $\ell_+$ denote the vector of edge lengths of graph $\Gamma_+(\ell)$, and $\ell_0$ denote the vector of edge lengths of $\Gamma_0(\ell)$, that is, $\ell=(\ell_+,\ell_0)$.  Next, since all components of $\ell$ are positive, the spaces $\partial\Gamma(\ell)$, $ \partial\Gamma_{+}(\ell)$, and $\partial\Gamma_{0}(\ell)$ do not depend on $\ell$. We therefore drop $\ell$ and  write $\partial\Gamma$, $ \partial\Gamma_{+}$, and $\partial\Gamma_{0}$ respectively. Then, in particular,  $\partial\Gamma=\partial\Gamma_{+}\cup\partial\Gamma_{0}$ and
\begin{align}
&L^2(\partial\Gamma)=L^2(\partial\Gamma_{+})\oplus L^2(\partial\Gamma_{0})\label{vv15},
\end{align}
where $L^2$ spaces correspond to the discrete Dirac measure with support $\cup_{e\in \cE} \{a_e, b_e\}$.
We notice that all spaces in \eqref{vv15} are finite-dimensional since
$\Gamma$ is a compact graph. Let $P_+$ be the orthogonal projection
acting in $L^2(\partial\Gamma)$ with
$\ran(P_+)=L^2(\partial\Gamma_+)\oplus \{0\}$, and let
$P_0:=I_{L^2(\partial\Gamma)}-P_+$. We recall the notation
$\dL^2(\partial\Gamma):=L^2(\partial\Gamma)\oplus L^2(\partial\Gamma)$
and we write $\prescript{d}{}P$ for the operator $P\oplus P$ acting in
$\dL^2(\partial\Gamma)$.
In particular, for the symplectic form \eqref{vv9}, one has
\begin{equation}
  \omega(u,v)=\omega(\dP_+u,\dP_+v)+\omega(\dP_0u,\dP_0v), \label{aa19}
\end{equation}
for all $u,v \in\, \dL^2(\partial\Gamma)$. 

To complete the setting, let us define the Schr\"odinger operators corresponding to each lengths vector $\ell\in\bbR_{>0}^{|\cE|}$. To this end,  let us fix a family of potentials   $q^{\ell}\in L^{\infty}(\Gamma(\ell);\bbR)$ corresponding to the graphs with positive edge lengths, and the limiting potential $q^{\wti\ell}\in L^{\infty}(\Gamma(\wti\ell);\bbR)$ satisfying 
\begin{equation}
\begin{split}\label{pot}
&\|q^{\ell}\|_{ L^{\infty}(\Gamma(\ell);\bbR)}=\cO(1)\text{\ as\ }\ell\rightarrow\wti\ell,\\
\sup\limits_{y\in [0,1]}|q_e^{\ell}(\ell_e y)&-q_e^{\wti\ell}(\wti\ell_e y)|=o(1) \text{\ as\ }\ell_e\rightarrow \wti\ell_e, \text{\ for all}\ e\in \cE_{+}.
\end{split}
\end{equation}
These conditions hold, for example, if the family $q^{\ell}$ is obtained by rescaling a fixed potential. Next,  we fix a Lagrangian plane 
\begin{equation}\label{3.9}
\cL\subset\, \dL^2(\partial\Gamma).
\end{equation}
For $\ell\in \bbR_{>0}^{|\cE|}$ let $H(\cL, \ell)$ denote the self-adjoint Schr\"odinger operator acting in $L^2(\Gamma(\ell))$ and given by the formulas
\begin{align}
\begin{split}	\label{c6}
&H({\cL}, \ell):=-\frac{d^2}{dx^2}+q^{\ell}(x),\\
\dom\big(H({\cL},&\,\ell)\big)=\{f\in \hatt{H}^2(\Gamma(\ell)): \tr^{\ell} f \in\cL\},
\end{split}
\end{align}
where the trace operator $\tr^{\ell}=(\gamma_D^{\ell}, \gamma_N^{\ell})^{\top}$ acts from $\hatt{H}^2(\Gamma(\ell))$ to $L^2(\partial\Gamma)\oplus L^2(\partial\Gamma)$ as indicated in \eqref{2.4new}. In particular, the norm of $\tr^{\ell}$ depends on $\ell$.
The resolvent of $H(\cL,\ell)$ is denoted by
\begin{equation}\label{aa6}
R(\cL, \ell, z):=\left(H(\cL, \ell)-z\right)^{-1},\ z\in\C\setminus\spec(H(\cL, \ell)).
\end{equation}

\section{Main Results}
\label{sec:main_results}
In this section we collect the statements of our main results together
with examples that illustrate their application.  The proofs will be
provided in subsequent sections.

First, we define an operator $H(\wti\cL, \wti \ell)$ on the
graph $\Gamma( \wti\ell)$, which will serve as the limiting operator for $H(\cL,\ell)$ as $\ell\rightarrow\wti\ell$. The definition is motivated by the heuristic
observation, made in the Introduction, that the limiting boundary conditions should be of the
form~\eqref{eq:equal_on_vanishing}.

\begin{theorem}\label{a30.1}
  Assume that $\cL\subset\, \dL^2(\partial\Gamma)$ is a Lagrangian
  plane with respect to symplectic form $\omega$,
  cf. \eqref{vv9}-\eqref{vv10}. Let
  \begin{equation}\label{u6}
    \wti\cL := \{ (\phi_1|_{\partial \Gamma+},\phi_2|_{\partial
      \Gamma+}): 
    (\phi_1,\phi_2) \in \cL \cap (D_0\oplus N_0)\},
  \end{equation}
  where
  \begin{align}
    \label{eq:D0_def}
    D_0&=\{\phi_1\in  L^2(\partial\Gamma): \phi_1(a_e)=\phi_1(b_e),\
         e\in\cE_0 \},\\
    \label{eq:N0_def}      
    N_0&=\{\phi_2\in  L^2(\partial\Gamma): \phi_2(a_e)=-\phi_2(b_e),\  e\in\cE_0 \}.
  \end{align}
  Then $\wti\cL$ is a Lagrangian plane in $\dL^2(\partial\Gamma_+)$
  with respect to the symplectic form $\omega_{\Gamma_+}$ obtained by
  restricting $\omega$ to $\dL^2(\partial\Gamma_+)$.
  Therefore, the operator $H(\wti \cL, \wti \ell)$ acting in $L^2(\Gamma(\wti\ell))$ and given by
  \begin{align}
    \begin{split}\label{u3}
      &H({\wti\cL},\wti \ell):=-\frac{d^2}{dx^2}+q^{\wti\ell},\\
      &\dom\big(H({\wti\cL},\wti\ell)\big)=\{f\in \hatt{H}^2(\Gamma(\wti\ell)): \tr^{\wti\ell}(f)\in\wti\cL\},
    \end{split}	
  \end{align}
  is self-adjoint.
\end{theorem}
\begin{proof}
  The proof, based on linear symplectic reduction, is provided on page
  \pageref{pga30.1}.
\end{proof}

The main result of this paper is the convergence of the spectral projections of the
self-adjoint Schr\"odinger operators $H(\cL,\ell)$ to those of
$H(\wti\cL,\wti\ell)$.  It will be established under the following
condition.

\begin{condition}[Non-resonance Condition]\label{newhyp}
  Suppose that for all $(\phi_1,
  \phi_2)\in\cL\cap(D_0\oplus N_0)$ such that $\phi_1|_{\partial
    \Gamma+}=\phi_2|_{\partial \Gamma+}=0$ one has $\phi_1=0$. 
\end{condition}

Informally, this condition says that if a function from the domain of
$H$ is small on non-vanishing edges, then its value (but not,
necessarily, its derivative) should also be small on the vanishing
edges.  Let us emphasize the striking similarity between
Condition~\ref{newhyp} and the definition of $\wti\cL$ in equation
(\ref{u6}).  As explained in Remark~\ref{rem:generic}, this condition
is generic among all self-adjoint
Schr\"odinger operators on $\Gamma$ (as parameterized by the
Lagrangian planes $\cL$).

Condition~\ref{newhyp} is also easy to check on important classes
of graphs.  The first class consists of the graphs with scale
invariant conditions, in which case Condition~\ref{newhyp} is also
necessary.

\begin{lemma}
  \label{y21} 
  Suppose that the Robin part of $H(\cL,\ell)$ is absent, that is,
  $P_R=0$ in \eqref{new5}. Then Condition~\ref{newhyp} holds if and
  only if the zero function is the only function satisfying the
  boundary conditions $\tr(f)\in \cL$, that is constant on each edge
  of $\Gamma_0$ and vanishes on $\Gamma_+$.
\end{lemma}
\begin{proof}On page \pageref{pgy21}. \end{proof}
The second class includes connected graphs with a continuity condition imposed at
every vertex.

\begin{lemma}\label{1a}
  Suppose that every vanishing edge $e\in\cE_0$ belongs to a path
  $\cP_e$ that contains at least one non-vanishing edge, and along
  which the function  $|f|$ is continuous for every $f\in\dom\big(H(\cL,\ell)\big)$.
  Then Condition~\ref{newhyp} holds.
\end{lemma}
\begin{proof}On page \pageref{pgy21}. \end{proof}

In order to formulate our results on spectral convergence, let us
introduce quasi-unitary operators $\cJ_{\ell}$ which lift the
functions defined on the limiting graph $\Gamma(\wti\ell)$ to the
approximating graph $\Gamma(\ell)$.  This is achieved by linear
scaling on the edges of $\Gamma_+$ and by extending functions by zero
on the edges of $\Gamma_0$, that is by defining
$\cJ_{\ell} \in \cB\left(L^2(\Gamma(\wti\ell)),
  L^2(\Gamma(\ell))\right)$ as follows:
\begin{equation}
  (\cJ_{\ell}f)(x)=\sum_{e\in\cE_+}
  \chi_{e}(x)\sqrt{\frac{\wti{\ell}_e}{\ell_e}}f\left(\frac{x\wti{\ell}_e}{\ell_e}\right),
  \quad x\in\Gamma(\ell),\label{d5}
\end{equation}
where $\chi_{e}(\cdot)$ is the characteristic function of
$e\subset\Gamma(\ell)$.  We remark that $\cJ_{\ell}^*\cJ_{\ell}$,
where $\cJ_{\ell}^*$ denotes the adjoint operator, is identity on
$L^2(\Gamma(\wti\ell))$, see Theorem~\ref{uu77} for details.

\begin{theorem}[Convergence of resolvents]\label{cc1}
  Assume Condition~\ref{newhyp}. Then, as $\ell\to\wti\ell$,
  \begin{equation}
    \label{eq:resolvent_convergence_AB}
    \begin{split}
      &\Big\|
      \cJ_{\ell}R(\wti\cL,\wti\ell,z) - R(\cL,\ell,z)\cJ_{\ell}
      \Big\|_{\cB\left(L^2(\Gamma(\wti\ell)),L^2(\Gamma(\ell))\right)}
      \to 0,\\
      &\big\|(I_{L^2(\Gamma(\cG;\ell))}-\cJ_{\ell}\cJ_{\ell}^*)
      R(\cL,\ell,z)\big\|_{\cB(L^2(\Gamma(\cG;\ell)))}
      \to 0,
    \end{split}
  \end{equation}
  where $R(\cL,\ell,z)$ and $R(\wti\cL,\wti\ell,z)$ denote the
  resolvents of the respective operators.
\end{theorem}

\begin{proof}On page \pageref{uu77} as a combination of
  Theorems~\ref{uu77} and \ref{u7}.
\end{proof}

An immediate corollary of the convergence of resolvents is
convergence of spectra.

\begin{theorem}[Spectral convergence]\label{thm:spectral_convergence}
  Assume Condition~\ref{newhyp} holds.  Then 
  \begin{equation}\label{eq:spectral_convergence}
    \spec(H(\cL,\ell))\rightarrow \spec(H(\wti\cL,\wti\ell))
    \ \text{as}\  \ell\rightarrow\wti\ell,
  \end{equation}
  in the Hausdorff sense for multisets.  Namely, if $\lambda_0$ has multiplicity $m\in\{0,1,2,\ldots\}$ in the
  multiset $\spec(H(\wti\cL,\wti \ell))$ then for all sufficiently
  small $\varepsilon>0$ there
  exists $\delta=\delta(\varepsilon, \lambda_0)>0$ such that
  \begin{equation}\label{c9}
    \card\left( \spec(H(\cL,\ell))\cap B(\lambda_0, \varepsilon)\right) = m
    \text{\ whenever } |\ell-\wti\ell|<\delta.
  \end{equation} 
  Furthermore, eigenspaces converge in the following sense,
  \begin{equation}
    \label{eq:projectors_convergence_AB}
    \begin{split}
      &\Big\|\cJ_{\ell}\chi(
      H(\wti\cL,\wti\ell))-\chi(H(
      \cL,\ell))\cJ_{\ell}\Big\|_{\cB\left(L^2(\Gamma(\wti\ell)),
      	L^2(\Gamma(\ell))\right)} \to 0,\\ 
      &\big\|\left(I_{L^2(\Gamma(\cG;\ell))}-\cJ_{\ell}\cJ_{\ell}^*\right)
      \chi(H(
      \cL,\ell))\big\|_{\cB(L^2(\Gamma(\cG;\ell)))}
      \to 0,\\
    \end{split}
  \end{equation}
  where $\chi(H(\cL,\ell))$ and
  $\chi( H(\wti\cL,\wti\ell))$ denote the spectral projections
  of the respective operators onto an interval $(a,b)$ with
  $a,b\in\bbR\setminus\spec(H(\wti\cL,\wti\ell))$.
\end{theorem}

\begin{proof}
  On page \pageref{pgspec}
\end{proof}

In the case of the Laplace operator with scale invariant vertex
conditions we show that Condition~\ref{newhyp} is not only sufficient
but also necessary for the spectral convergence to hold.

\begin{theorem}
  \lb{scale} Assume that the Robin part of $H(\cL,\ell)$ is absent ,
  that is, $P_R=0$ in \eqref{new5} and that $q^\ell\equiv 0$ .  Then
  \eqref{eq:spectral_convergence} holds \emph{if and only if}
  Condition~\ref{newhyp} is fulfilled.
\end{theorem}

\begin{proof}
On page \pageref{pgscale}.
\end{proof}

While Condition~\ref{newhyp} is convenient to use (see numerous
examples below), it will not be used directly in the proofs.  Instead
we will need a more technical result: a uniform bound on the resolvent
of $H(\cL, \ell)$ as an operator from $L^2(\Gamma(\ell))$ to
$L^{\infty}(\Gamma(\ell))$ which follows from Condition~\ref{newhyp}.
In fact, it is this bound that implies the conclusion of Theorem
\ref{cc1}. We explore this bound in the following two theorems.

\begin{theorem}\label{ss1} Recall \eqref{3.9}--\eqref{aa6}. Then the following statements are equivalent:
  \begin{itemize}
  \item[(i)] There exists a constant $c>0$, independent of $\ell$, such that
    \begin{equation}
      \label{eq:resolvent_bound}
      \|R(\cL, \ell, \bfi)\|_{\cB{(L^2(\Gamma(\ell)),
          L^{\infty}(\Gamma(\ell)))}} < c,
    \end{equation}
    for all $\ell$
    sufficiently close $\wti\ell$.
  \item[(ii)] There exists a constant $c>0$, independent of $\ell$, such that
    \begin{equation}\no
      \big\|\chi_eR(\cL,\ell, \bfi)\big\|_{\cB(L^2(\Gamma(\ell)))} <
      c\sqrt{\ell_e}\quad  \text{for each\  }e\in\cE_0,
    \end{equation}
    for all $\ell$
    sufficiently close $\wti\ell$.
  \item[(iii)] There exists a constant $c>0$, independent of $\ell$ and $f$, such that 
    \begin{equation}\label{ss4}
      \|f\|_{L^{\infty}(\Gamma(\ell))}^2
      \leq  c \left(\| f\|_{L^2(\Gamma(\ell))}^2+ \| f''\|_{L^2(\Gamma(\ell))}^2\right),\ f\in\dom(H(\cL,\ell)),
    \end{equation} 
    for all $\ell$
    sufficiently close $\wti\ell$
  \end{itemize}
  Moreover, if one of the above statements holds then for some
  constant $c>0$, independent of $\ell$ and $f$, one has
  \begin{equation}\label{qqq1}
    \| f'\|_{L^2(\Gamma(\ell))}^2\leq c \left(\| f\|_{L^2(\Gamma(\ell))}^2+ \| f''\|_{L^2(\Gamma(\ell))}^2\right),\  f\in\dom(H(\cL,\ell)),
  \end{equation} 
  and
  \begin{align}
    \begin{split}\label{b8}
      &\| R(\cL, \ell, \bfi)\|_{\cB\big(L^2(\Gamma(\ell)),
        \hatt{H}^2(\Gamma(\ell))\big)} < c,
    \end{split}
  \end{align}
 for all $\ell$
 sufficiently close $\wti\ell$.
\end{theorem}
\begin{proof}
On page \pageref{pgss1}.
\end{proof}
\begin{theorem}
  \label{y1}  
  Condition~\ref{newhyp} implies statements (i)-(iii) of Theorem~\ref{ss1}. Furthermore, if the Robin part of $H(\cL,\ell)$ is absent,
  that is, $P_R=0$ in \eqref{new5}, then (i)-(iii) of Theorem~\ref{ss1} are equivalent to Condition~\ref{newhyp} 
\end{theorem}
\begin{proof}
On page \pageref{pgy1}.
\end{proof}

To illustrate our results we will now discuss several examples of
graphs with shrinking edges.  We start with the most basic example
where there is no spectral convergence.

\begin{example}[Shrinking Neumann interval]
  \label{5.5}
  In this example we consider a disconnected two edge graph
  $\Gamma=\{e_N, e_D\}$ and the Laplace operator subject to Neumann
  boundary conditions on $e_N$ and to Dirichlet boundary conditions on
  $e_D$. The spectrum of such quantum graph is given by
  \begin{equation}\label{sp}
    \{0\}\cup\left\{  \left(\frac{\pi k_1}{\ell_D}\right)^2,
      \left(\frac{\pi k_2}{\ell_N}\right)^2: k_1\in\bbN, k_2\in
      \bbN\right\}. 
  \end{equation}

  Now let $\ell_N \to 0$ while $\ell_D=1$.  Condition~\ref{newhyp} (in the form of
  Lemma~\ref{y21}) fails: the function equal to 1 on $e_N$ and
  $0$ on $e_D$ satisfies the boundary conditions for all $\ell_N$.  This
  function gives rise to eigenvalue $0$ which is \emph{not} present in
  the spectrum of $H(\wti\cL,\wti\ell)$ defined according to
  (\ref{u6}). The latter operator is simply the Dirichlet Laplacian on
  the interval $e_D$ whose spectrum is
  \begin{equation}\label{sp_lim}
    \left\{\left(\frac{\pi k_1}{\ell_D}\right)^2: k_1\in\bbN \right\}. 
  \end{equation}

  A slight variation of this example is the same graph with
  $\ell_D \to 0$, $\ell_N=1$, which does satisfy
  Condition~\ref{newhyp}. The limiting operator is the Neumann
  Laplacian on the interval $e_N$ whose spectrum is the limit of the
  set in (\ref{sp}) as $\ell_D \to 0$.
\end{example}

Despite its simplicity, the above example illustrates the common
mechanism of convergence failure: presence of an eigenfunction whose
support is shrinking to zero.  The following example shows that there
are connected graphs with similar features.

\begin{figure}[h]
  \includegraphics[scale=1.2]{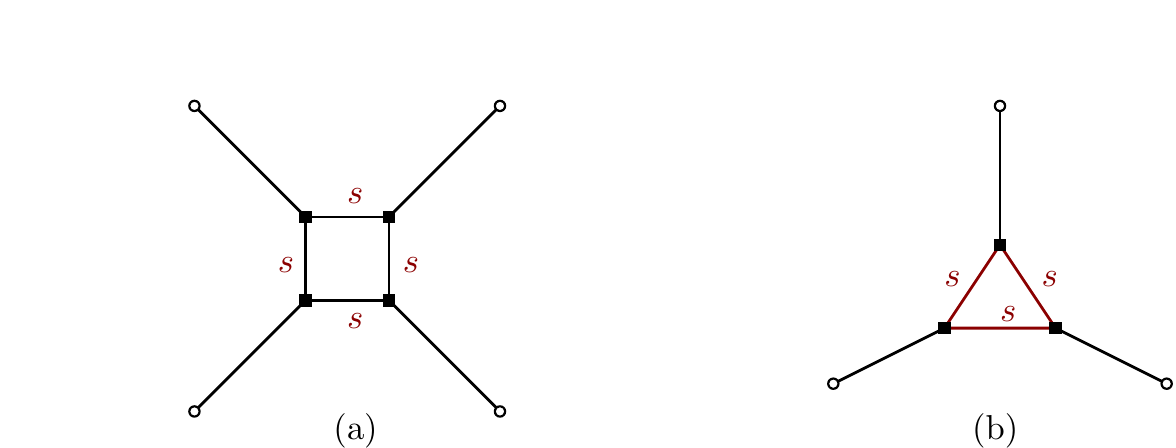}
  \caption{Two similar graphs with different convergence outcomes as
    $s\to0$: the spectrum of (a) does not converge while that of (b) does.  Empty circles
    denote Dirichlet conditions and full squares denote anti-Kirchhoff
    conditions, cf. \eqref{gh11}.}
  \label{fig:div_resolvent}
\end{figure}

\begin{example}
  \label{ex:divergence_antikirchhoff}
  Consider the graph in Fig.~\ref{fig:div_resolvent}(a), equipped with
  anti-Kirchhoff conditions, cf. \eqref{gh11}. Condition \ref{newhyp} fails by Lemma~\ref{y21} since there is a function
  equal to $+1$ on vertical vanishing edges, $-1$ on horizontal
  vanishing edges and zero on all non-vanishing edges.  This is an
  eigenfunction with eigenvalue zero whose support is the vanishing
  part.

  We point out that the spectral convergence (or lack thereof) depends
  not only on the boundary conditions but also on the topology of the
  graph.  It is easy to see that the graph in
  Fig.~\ref{fig:div_resolvent}(b), despite having the same vertex
  conditions as Fig.~\ref{fig:div_resolvent}(a), satisfies the
  conditions of Lemma~\ref{y21}: the only function, constant on each
  edge and equal to zero on the non-vanishing edges must be zero on
  the whole graph.
\end{example}


Are the eigenfunctions of eigenvalue 0 the only ones to cause such
problems?  In general, the answer is no.  Let us start with a related
question: suppose the whole graph is scaled by $s$ as $s \to 0$.  Weyl's law
dictates that the bulk of the eigenvalues grow at the rate $1/s^2$.
If the vertex conditions are scale invariant, all eigenvalues of the
Laplacian get multiplied by $1/s^2$ and grow (except the eigenvalue
$0$).  But is the same true in general?

\begin{figure}
  \includegraphics[scale=1]{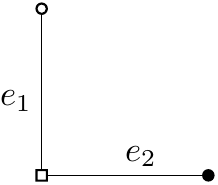}
  \caption{\label{y41} $\circ$ denotes Dirichlet vertex
    conditions, $\bullet$ denotes Neumann--Kirchhoff conditions, cf. \eqref{gh9},
    $\Box$ denotes vertex conditions
    given by (\ref{eq:hyperbolic_cond})}
\end{figure}

\begin{example}
  \label{ex:hyperbolic}
  Consider the graph consisting of two edges of length
  $\ell_1=\ell_2=s$ connected at one endpoint, see Figure \ref{y41}. Impose
  Dirichlet and Neumann conditions at  endpoints of degree one of edges
  $e_1$ and $e_2$, correspondingly.  At the vertex of degree 2 impose
  the conditions that we will call \emph{hyperbolic},
  \begin{equation}
    \label{eq:hyperbolic_cond}
    \begin{cases}
      \partial_{\nu} f_{e_1}(v)=-f_{e_2}(v),\\ 
      \partial_{\nu} f_{e_2}(v)=-f_{e_1}(v).
    \end{cases}
  \end{equation}
  This graph has vanishing volume but $-1$ remains an eigenvalue
  independently of $s$.  The corresponding eigenfunction is
  \begin{equation}
    \label{eq:hyperbolic_efun}
    f_{e_1}(x) = \sinh(x), \qquad f_{e_2}(x) = \cosh(x),    
  \end{equation}
  where on both edges the point $x=0$ is at the vertex of degree one.
\end{example}

We now turn this into an example of a \emph{connected} graph with some
non-vanishing edges.

\begin{figure}[h]
  \includegraphics[scale=1.0]{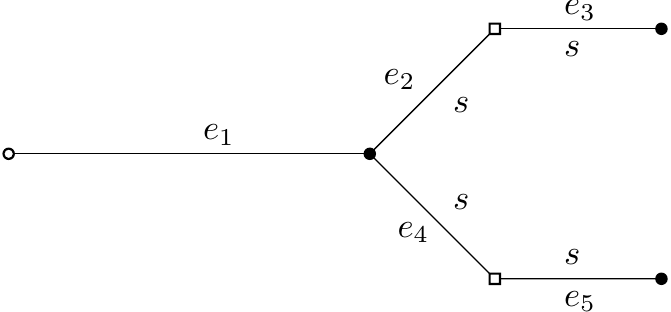}
  \caption{$e_1$ is of length 1, $e_k$ is of length $s$ for
    $k\in\{2,3,4,5\}$, $\circ$ denotes Dirichlet conditions, $\bullet$
    denotes Neumann--Kirchhoff conditions, $\Box$ denotes vertex given
    by~\eqref{eq:hyperbolic_cond}.}
  \label{hbc}
\end{figure}

\begin{example}
  \label{ex:hyperbolic_fork}
  Consider the graph shown in Fig.~\ref{hbc}.  The lengths of the edges
  are $\ell_1=1$ and $\ell_2=\ell_3=\ell_4=\ell_5=s \to 0$.  There is
  an eigenfunction with eigenvalue $-1$ for every $s>0$:
  \begin{equation}
    \label{eq:hyperbolic_fork_efun}
    f_{e_1}(x) \equiv 0, \,
    f_{e_2}(x) = \sinh(x),  \,
    f_{e_3}(x) = \cosh(x),  \,
    f_{e_4}(x) = -\sinh(x),  \,
    f_{e_5}(x) = -\cosh(x).
  \end{equation}
  Using \eqref{eq:equal_on_vanishing}, it is easy to see that
  $H(\wti\cL, \wti\ell)$ is the edge $e_1$ with Dirichlet conditions.
  Thus every approximating graph has an eigenvalue $-1$ while its
  would-be limit is a strictly positive operator.
\end{example}

Let us now explore some examples where the spectral convergence holds.

\begin{example}[Tadpole graph with a vanishing loop]
  \label{ex:tadpole}

  \begin{figure}
    \hspace{-2cm}
    \includegraphics[scale=1]{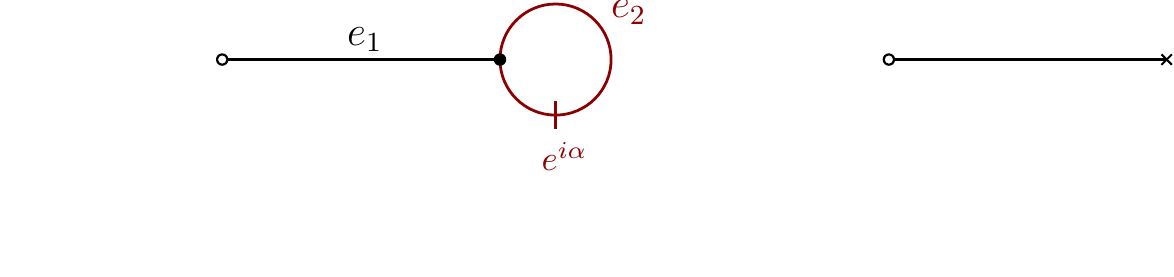}
    \includegraphics[width=5in]{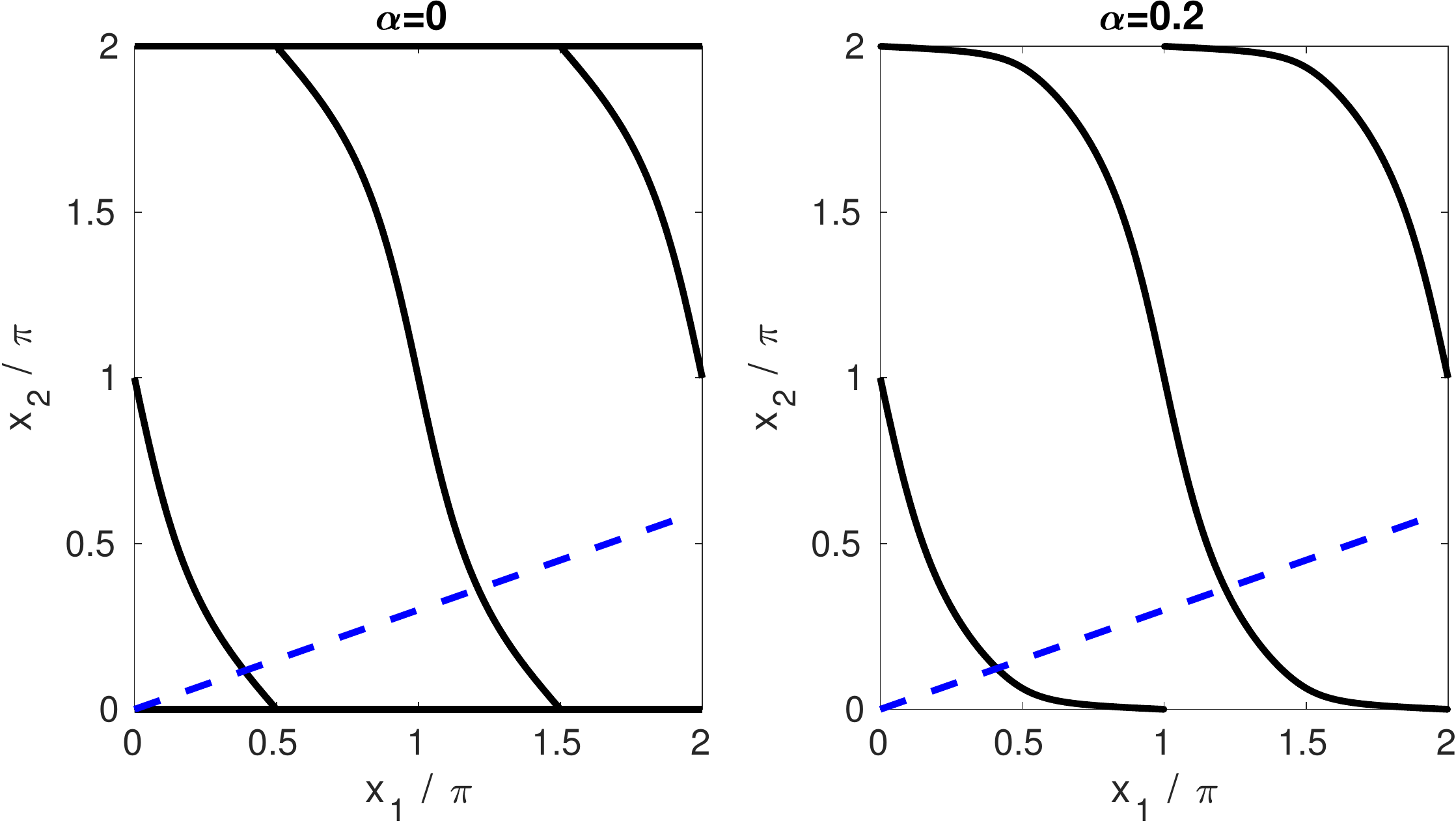}
    \caption{\label{tadpole}Top panel: $\circ$ denotes Dirichlet
      vertex conditions, $\bullet$ denotes Kirchhoff conditions , $\times$
      denotes Dirichlet conditions if $\alpha\not=0$ mod ($2\pi$) and
      Kirchhoff conditions if $\alpha=0$ mod ($2\pi$). \newline Bottom
      panel: The secular manifold and the Barra--Gaspard flow for two
      values of $\alpha$.}
  \end{figure}
  Consider the graph consisting of an edge and a loop attached to one
  of its endpoints, see Figure \ref{tadpole}.  We impose
  Neumann--Kirchhoff conditions at the attachment point and the
  Dirichlet condition at the other endpoint.  We assume the magnetic
  flux $\alpha$ is threading the loop.  The magnetic field is realized
  as the condition
  \begin{equation}
    \label{mag_condition}
    f(c+) = e^{i\alpha} f(c-) \qquad \partial_\nu f(c+) =
    -e^{i\alpha} \partial_\nu f(c-)
  \end{equation}
  at an arbitrary point $c$ on the loop.  The derivative is taken in
  the direction away from $c$ according to our convention; this leads
  to the minus sign in \eqref{mag_condition}. Let $\ell_1=1$ be the
  length of the edge and $\ell_2=s$ be the length of the loop.  The
  spectral convergence, as $s\rightarrow 0$, holds by Lemma~\ref{1a}
  and Theorem~\ref{thm:spectral_convergence}.  However, the limiting
  operator depends on whether $\alpha=0$ or not.

  It is interesting to explore this difference from the point of view
  of the secular manifold.  Following a well-known procedure \cite{B17}, the eigenvalues
  $\lambda=k^2>0$ of this graph can be found as the solutions of the
  secular equation $F(kl_1, kl_2; \alpha)=0$, where, in this case, the
  secular function $F$ is given by
  \begin{equation}
    \label{sec_fun_tadpoleN}
    F(x_1, x_2; \alpha) 
    = -2 \sin x_1 \left(\cos x_2 - \cos(\alpha) \right) 
    - \cos x_1 \sin x_2.
  \end{equation}
  
  To understand the behavior of eigenvalues, we follow Barra--Gaspard,
  cf. \cite{BG}, and
  visualize them as the intersections of the straight line
  $[k l_1, k l_2]$, $k\in(0,\infty)$ with the analytic  variety
  \begin{equation}
    \label{sec_surf_def}
    \Sigma_\alpha = \{ (x_1,x_2) \in \mathbb{R}^2 : F(x_1,x_2;
    \alpha)=0\},
  \end{equation}
  usually referred to as \emph{secular manifold}.  This convenient
  characterization is available only for graphs with scale invariant
  vertex conditions and zero potential.  Both the line and the secular
  manifold $\Sigma_\alpha$ for two values of $\alpha$ (zero and
  non-zero) are illustrated in Figure~\ref{tadpole}.  Since we are
  setting $l_1=1$, the values of $k$ can be read as the $x$-coordinate
  of the intersection points.

  The structure of secular manifold undergoes a significant change
  from $\alpha=0$ to $\alpha \neq 0$.  When $\alpha=0$, the secular
  manifold on the torus is a union of a smooth curve and the line
  $x_2 = 0$.  When $\alpha \neq 0$, there are two smooth curves (which
  related by a shift of $\pi$ in $x_1$ direction).

  Suppose that the slope of the dashed lines in Figure \ref{tadpole}
  is equal to $s$. Then as $s\rightarrow 0$, the first intersection
  point converges to $(\pi/2,0)$ when $\alpha=0$ mod ($2\pi$) and to
  $(\pi,0)$ otherwise. That is, the first intersection point tends to
  the first eigenvalue of the Neumann--Dirichlet interval if
  $\alpha=0$ mod ($2\pi$) and to the first eigenvalue of
  Dirichlet--Dirichlet interval otherwise.

  If, instead of contracting the loop, we contract the edge, our
  results dictate that the loop will get the Dirichlet conditions at
  the (former) attachment point.  This disconnects the loop into an
  interval of length $l_2$ with Dirichlet endpoints and the spectrum
  $k_n = \pi n / l_2$.  The result is independent of $\alpha$ (the
  magnetic field on an interval can be removed by a gauge
  transformation) and can be seen both from Figure~\ref{tadpole} (the
  dashed line is getting close to vertical) or from setting $x_1=0$ in
  the secular function, equation \eqref{sec_fun_tadpoleN}.

  Finally, we remark that simply setting the relevant $x=0$ does not
  always produce the correct secular function for the limiting
  problem: as observed in \cite{ABB}, we get identically zero if we set
  $x_2=0$ for the loop with no magnetic field ($\alpha=0$).
\end{example}

\begin{example}[A vanishing cycle in a graph with Neumann--Kirchhoff
  conditions]
  Consider the tetrahedron graph (complete graph on 4 vertices, $K_4$)
  with one vertex turned into a triangle.  We will be contracting the
  triangle into a single vertex, see Figure~\ref{tetra}, scaling it by
  $s \to 0$.  We notice that the assumption of Lemma~\ref{1a} is
  satisfied, hence, the spectral convergence holds.

  \begin{figure}
    \hspace{-2cm}
    \includegraphics[scale=1]{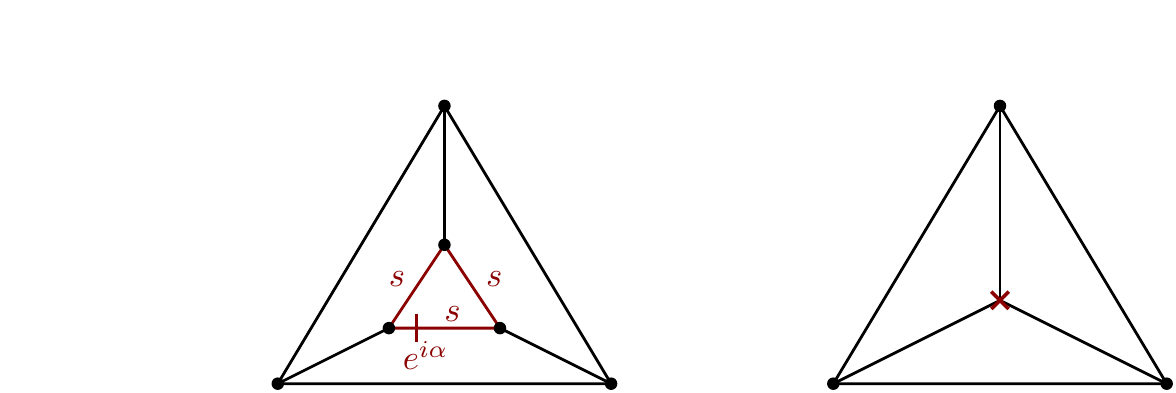}
    
    \vspace{0.5cm}
    \includegraphics[scale=0.7]{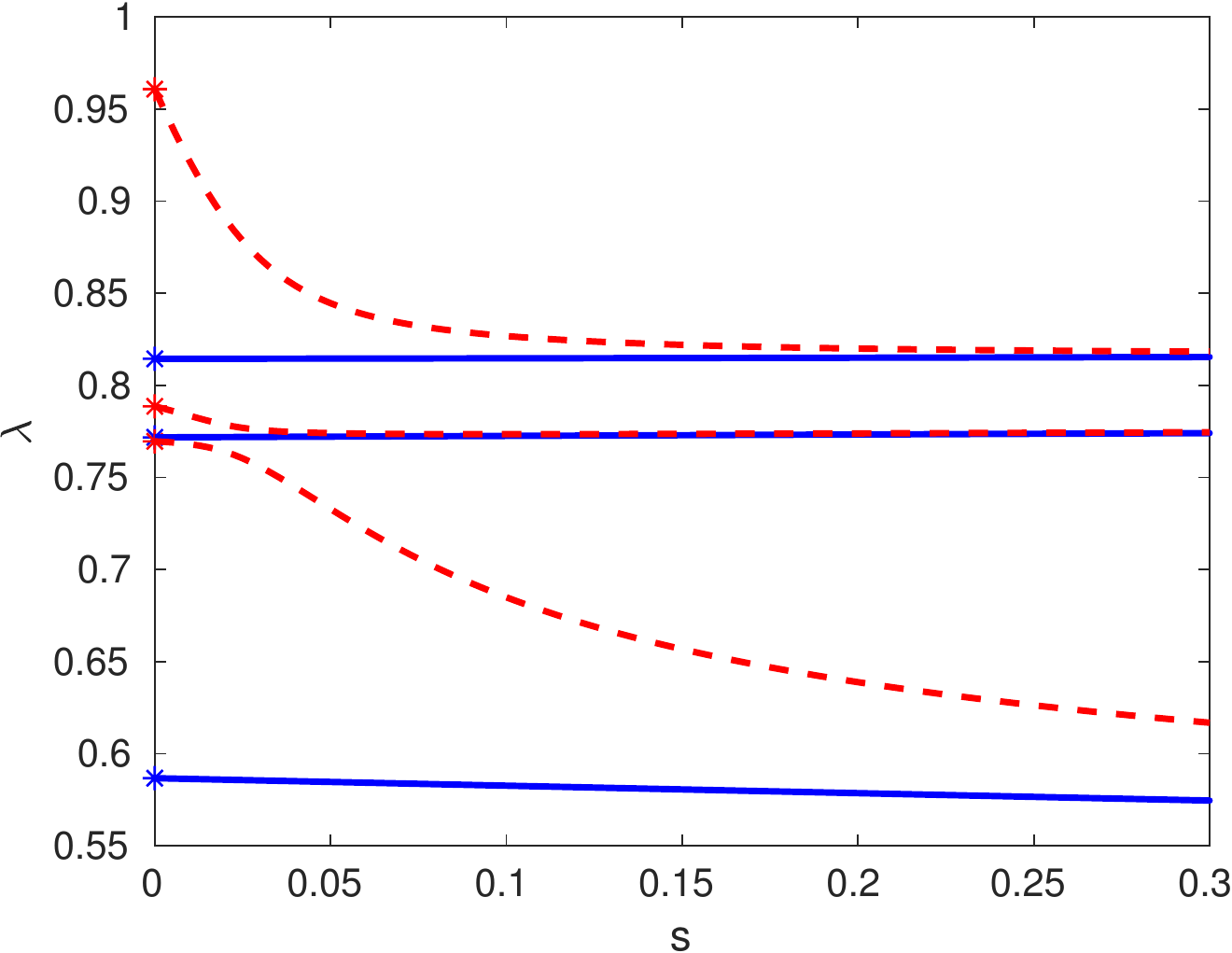}
    \caption{\label{tetra} Bottom panel:
      Numerical calculation of the spectrum of a graph with a
      cycle of length 3 contracting into a single vertex.  Blue curves
      correspond to no magnetic field, red lines correspond to a small
      flux threading the cycle.  The limiting eigenvalue displayed as
      stars at $s=0$ were calculated from the limiting vertex being
      supplied with Neumann--Kirchhoff (solid blue line) and Dirichlet
      (dashed red line) conditions.}
  \end{figure}

  We will thread magnetic flux $\alpha$ through the small triangle,
  realized as imposing conditions \eqref{mag_condition} on one of its
  edges.  The limit predicted by our results depends on the value of
  the flux.  For zero flux we simply recover Neumann--Kirchhoff
  conditions at the limiting vertex.  When flux is non-zero (modulo
  $2\pi$), the limiting conditions are Dirichlet which effectively
  disconnects the three edges at the central vertex.  These results are
  confirmed by the agreement between the results for small $s$ and the
  limiting graph computations, shown in Figure~\ref{tetra}.
\end{example}

\begin{example}
  \label{ex:hyperbolic_long}
  In a slight modification of Example~\ref{ex:hyperbolic}, we consider
  the graph displayed in Figure \ref{y41}, but with edge $e_2$ now
  having constant length $1$ while $e_1$ is shrinking. In this setting Condition \ref{newhyp} is satisfied and
 the spectral convergence holds.
\end{example}

\section{Lagrangian and Symplectic Subspaces}
\label{sec5}

The purpose of this section is to provide proof of the results that
make heavy use of symplectic geometry, namely Proposition~\ref{b1},
Theorem~\ref{a30.1} and Lemmas~\ref{y21} and \ref{1a}.  We start by
collecting the basic facts and definitions (see, for example, \cite{McS} for
further information).

\begin{definition}\label{a1}
  Let $n\in\bbN$ and $S$ be a complex linear space of dimension $2n$.
  A form $\omega:S\times S \rightarrow \C$ is called symplectic if the
  following holds:
  \begin{itemize}
  \item[(i)]$\omega$ is {\it sesquilinear}, that is,
    $\omega(\alpha x+\beta
    y,z)=\overline{\alpha}\omega(x,z)+\overline{\beta}\omega(y,z)$ and
    $\omega(z, \alpha x+\beta y)=\alpha\omega(z,x)+\beta\omega(z,y)$,
    for all $x,y,z\in S$ and $\alpha, \beta\in\C$
  \item[(ii)] $\omega$ is {\it skew-Hermitian}, that is,
    $\omega(x,y)=-\overline{\omega(y,x)}$, for all $x,y\in S$,
  \item[{(iii)}] $\omega$ is {\it nondegenerate}, that is,
    if $\omega(x, y)=0$ for all $y\in S$, then $x=0$.
  \end{itemize}
  The pair $(S,\omega)$ is called a \emph{symplectic space}.
\end{definition}

Let $\omega$ be a symplectic form on $\C^{2n}$ and $V\subset \C^{2n}$
be a linear subspace. The {\it annihilator} of $V$ is denoted by
$V^{\circ}$ and defined by the formula
\begin{equation}\label{a2}
V^{\circ}:=\{x\in\C^{2n}: \omega(x,y)=0 \text{\ for all\ }y\in V \}.
\end{equation}
Since the form $\omega$ is nondegenerate one has \cite[Lemma~2.2]{McS}
\begin{equation}
  \dim(V)+\dim (V^{\circ}) = 2n,\qquad
  \label{eq:dbl_annihilator}
  \left(V^\circ\right)^\circ = V.
\end{equation}

\begin{definition}\label{a3}
  Let $\omega$ be a symplectic form on $\C^{2n}\times\C^{2n}$ and let
  $S,V,W,\cL \subset\C^{2n}$ be linear subspaces. Then $S$ is called
  \emph{symplectic} if $S^{\circ}\cap S=\{0\}$, $V$ is called
  \emph{isotropic} if $V\subset V^{\circ}$, $W$ is called
  \emph{co-isotropic} if $W^\circ \subset W$, and $\cL$ is called
  \emph{Lagrangian} if $\cL^{\circ}=\cL$.
\end{definition}
A subspace $S$ of a symplectic space $(\C^{2n},\omega)$ is
symplectic if and only if the restriction $\omega|_S$ of $\omega$ on
$S$ is a symplectic form on the linear space $S$ (in other words, $S$
is a symplectic subspace if and only if $(S, \omega|_S)$ is a
symplectic space).

The main use of the Lagrangian theory in this paper is to characterize
self-adjoint vertex conditions on a graph.

\begin{proof}[Proof of Proposition \ref{b1}]
  The second Green's identity yields
  \begin{align}
    &\langle H_{max}f,g \rangle_{L^2(\Gamma)}
      -\langle f,H_{max}g \rangle_{L^2(\Gamma)}
      =\omega(\tr f, \tr g).\label{b9.1}
  \end{align}
  where $f,g\in \hatt{H}^2(\Gamma)$.
  
  Let us assume that $H$ is a self-adjoint extension of $H_{min}$. The subspace 
  \begin{equation}\no
    \tr\big({\dom(H)}\big) \subset L^2(\partial \Gamma)\oplus L^2(\partial \Gamma),
  \end{equation}
  is isotropic since $\omega(\tr f, \tr g)=0$ whenever
  $f,g\in\dom(H)$. In order to show that it is maximal, we recall that $\ran(\tr) = L^2(\partial \Gamma)\oplus L^2(\partial \Gamma)$.
  Assume that $w\in\big(\tr\left({\dom(H)}\right)\big)^{\circ}$, then
  there exists $f\in \hatt{H}^2(\Gamma)$ such that $w=\tr f$. Then for
  any $g\in \dom(H)$, one has
  \begin{equation}\no
    \langle Hf,g \rangle_{L^2(\Gamma)}
    -\langle f,Hg\rangle_{L^2(\Gamma)}
    =\omega(\tr f, \tr g)=0,
  \end{equation}
  hence, $f\in\dom(H^*)=\dom(H)$. Therefore, the subspace
  $\tr\big({\dom(H)}\big)$ is Lagrangian. We complete the proof of
  injectivity in the first part of the statement of the proposition by
  noticing that if $H_k=H_k^*, k=1,2$ are two self-adjoint extensions
  of $H_{min}$ satisfying
  \begin{equation}\label{aa1}
    \tr\big({\dom(H_1)}\big)=\tr\big({\dom(H_2)}\big),
  \end{equation}
  then 
  \begin{equation}\label{aa2}
    H_k\subset H^*|_{\dom(H_1)+\dom(H_2)}=(H^*|_{\dom(H_1)+\dom(H_2)})^*, k=1,2.
  \end{equation}
  Since $H_k, k=1,2$ are self-adjoint operators, \eqref{aa2} yields
  \begin{equation}
    \dom(H_1)=\dom(H_2), \text{\ hence,\ }H_1=H_2.
  \end{equation}
  
  To prove the second assertion in the proposition, let us fix a
  Lagrangian plane
  $\cL\subset L^2(\partial \Gamma)\oplus L^2(\partial
  \Gamma)$.
  Clearly, the operator given by \eqref{b7} is symmetric. Furthermore,
  for arbitrary $h\in\dom(H({\cL})^*)$ and $g\in\dom(H({\cL}))$ one
  has
  \begin{equation*}
    0=\langle H(\cL)^*h,g \rangle_{L^2(\Gamma)}
    -\langle h,H({\cL})g \rangle_{L^2(\Gamma)}=\omega(\tr h, \tr g).
  \end{equation*}
  Therefore, $\tr h \in \cL^{\circ}=\cL$ and $h\in
  \dom(H(\cL))$. Hence, $H(\cL)$ is a self-adjoint operator.
\end{proof}

To establish Theorem~\ref{a30.1} we use a technique sometimes called
\emph{linear symplectic reduction}.

\begin{proposition}[see, for example, {\cite[Lemma I.2.7]{McS}} or
  {\cite[Proposition I.8.4]{LM}}]
  \label{prop:symplectic_reduction}
  Let $W$ be a co-isotropic subspace of the symplectic space
  $(\C^{2n}, \omega)$.  The \emph{reduced symplectic space} associated
  with $W$ is the space
  \begin{equation}
    \label{eq:reduced_space_def}
    \dot{W} = W / W^\circ
  \end{equation}
  with the symplectic form naturally induced by $\omega$.

  If $\cL$ is a Lagrangian subspace of $(\C^{2n}, \omega)$, then the
  projection of $\cL \cap W$ onto $\dot{W}$ is Lagrangian in
  $\dot{W}$.
\end{proposition}

\begin{proof}[Proof of Theorem \ref{a30.1}]
  \lb{pga30.1} In order to prove that $\wti\cL$ is a Lagrangian plane
  we let $W = D_0 \oplus N_0$ and investigate $\dot{W}$.
  We recall that
  \begin{equation}
    \label{eq:Wdef}
    W = \{(\phi_1,\phi_2) \in  \dL^2(\partial\Gamma):
    \phi_1(a_e)=\phi_1(b_e),\ 
    \phi_2(a_e)=-\phi_2(b_e),\ 
    e\in\cE_0 \}.
  \end{equation}
  Importantly, $\phi_1$ and $\phi_2$ take arbitrary values on the
  edges $e \in \cE_+$.  Explicit calculation shows that
  \begin{align}
    W^\circ &= \{(\phi_1,\phi_2) \in  \dL^2(\partial\Gamma):
    \phi_1|_{\partial \Gamma+} = \phi_2|_{\partial \Gamma+} = 0;\ 
    \phi_1(a_e)=\phi_1(b_e),\ 
    \phi_2(a_e)=-\phi_2(b_e),\ 
    e\in\cE_0\} \nonumber \\
    \label{eq:Wcirc}
    &= \{(\phi_1,\phi_2) \in  \dL^2(\partial\Gamma):
    \phi_1|_{\partial \Gamma+} = \phi_2|_{\partial \Gamma+} = 0; \}
    \cap W.
  \end{align}
  This shows that $W$ is co-isotropic and $W / W^\circ$ is naturally
  identified with $\dL^2(\partial \Gamma_+)$.  By the second part of
  Proposition~\ref{prop:symplectic_reduction}, $\wti\cL$ which is
  defined in~\eqref{u6} as the projection of $\cL \cap W$ to
  $\dL^2(\partial \Gamma_+)$ is Lagrangian in
  $\dL^2(\partial \Gamma_+)$.
  
\end{proof}

\begin{remark}
  \label{rem:generic}
  Using equation~\eqref{eq:Wcirc} we can succinctly write
  Condition~\ref{newhyp} as
  \begin{equation}
    \label{eq:newhyp_short}
    \cL \cap (D_0 \oplus N_0)^\circ \,\subset\, 0 \oplus L^2(\partial \Gamma_0).
  \end{equation}
  Note the similarity to the condition of transversality of $\cL$ and
  $D_0 \oplus N_0$, namely $\cL \cap (D_0 \oplus N_0)^\circ = 0$ (we
  used that $\cL=\cL^\circ$).  Transversality is generic in the
  Grassmanian of all Lagrangian planes $\cL$.  Therefore, our less
  restrictive Condition~\ref{newhyp} is also generic.
\end{remark}

\subsection{Geometry of Condition~\ref{newhyp}}

In this section we delve deeper into the meaning of
Condition~\ref{newhyp} and prove Lemmas~\ref{y21} and \ref{1a}.  To
approach Lemma~\ref{y21} we characterize scale invariant conditions in
terms of the Lagrangian plane $\cL$.

\begin{proposition}\label{gh1} 
  The vertex conditions, \eqref{new5}, for the operator $H(\cL)$ are
  scale invariant, that is, $P_R=0$, if and only if there exist
  subspaces $\cL_D\subset L^2(\partial \Gamma)$ and
  $\cL_N\subset L^2(\partial \Gamma)$ such that
  \begin{equation}\lb{gh2}
    \cL=\left\{(\phi_1,\phi_2)\in\,\dL^2(\partial \Gamma): 
      \phi_1\in\cL_D, \phi_2\in\cL_N \right\}.
  \end{equation}
\end{proposition}

\begin{proof}
  If $P_R=0$ then \eqref{gh2} holds with $\cL_D:=\ker (P_D)$,
  $\cL_N:=\ker (P_N)$.  Conversely, assuming \eqref{gh2} we will first establish that
  \begin{equation}\lb{gh3}
    \cL_D = \cL_N^\perp.
  \end{equation}
  Let us pick arbitrary $f\in\cL_N^{\perp}$ and notice that for all $\phi_1\in\cL_D$, $\phi_2\in\cL_N$ one has
  \begin{equation}
    \omega((f, 0), (\phi_1, \phi_2))=-\int_{\partial\Gamma}\overline{f}\phi_2=0.
  \end{equation}
  Since $\cL$ is Lagrangian, this yields $(f, 0)\in\cL$ and, in
  particular, $f\in\cL_D$.  Next, to prove
  $ \cL_D\subset \cL_N^{\perp}$ we observe that
  $(\phi_1,0), (0,\phi_2)\in\cL$ for all $\phi_1\in\cL_D$,
  $\phi_2\in\cL_N$, thus
  \begin{equation}\lb{gh5}
    0=\omega((\phi_1, 0), (0, \phi_2))=-\int_{\partial\Gamma}\overline{\phi_1}\phi_2.
  \end{equation}
  Let $P_D$, $P_N$ denote the orthogonal projections in
  $L^2(\partial \Gamma)$ with $\ker(P_D)=\cL_D$ and
  $\ker(P_N)=\cL_N$. Then
  \begin{equation}
    \label{eq:rangesP}
    \ran(P_D) \oplus \ran(P_N) = \cL_N \oplus \cL_D = L^2(\partial \Gamma)
  \end{equation}
  and by \eqref{b7}, \eqref{new5} and \eqref{gh3}
  \begin{equation}
    \dom(H(\cL))=\left\{f\in\hatt H^2(\Gamma)\Big|
      P_D \gamma_Df=0,\,P_N \gamma_Nf=0\right\}.
  \end{equation}
  Thus, $P_R=0$.
\end{proof}

\begin{proof}[Proof of Lemma~\ref{y21}] \lb{pgy21} Let us note that
  Condition \ref{newhyp} can be succinctly written as follows
  \begin{equation}\label{pq4}
    (\phi_1,\phi_2) \in \cL\cap (D_0\oplus N_0)\cap \ker(\,\dP_+)
    \quad\Rightarrow\quad \phi_1=0.
  \end{equation} 
  
  Suppose that the assumption of the Lemma
  holds yet Condition \ref{newhyp} is not satisfied. Then pick
  arbitrary
  $(\phi_1, \phi_2) \in \cL\cap (D_0\oplus N_0)\cap \ker(\,\dP_+)$
  with $\phi_1 \neq 0$ and define the following function
  \begin{equation}
    f:=\sum_{e\in\cE_0}\phi_1(a_e)\chi_e\not\equiv 0.
  \end{equation}	
  By construction, we have $\gamma_D f = \phi_1$.  Also, since the
  function is constant on every edge, $\gamma_N f = 0$.  Since our
  vertex conditions are scale invariant, by Proposition~\ref{gh1},
  $(\phi_1, \phi_2) \in \cL$ implies $(\phi_1, 0) \in \cL$ and
  therefore $f \in \dom(H(\cL,\ell))$, in contradiction to the
  assumption.

  Conversely, suppose that $f$ is a nonzero function constant on each
  edge satisfying the boundary conditions and such that
  $\supp(f)\subset \Gamma_{0}$. Then
  $\tr f\in \cL\cap (D_0\oplus N_0)\cap \ker(\,\dP_+)$ yet
  $\gamma_D f \neq 0$ and therefore the choice
  $(\phi_1,\phi_2) = \tr f$ falsifies Condition~\ref{newhyp}.
\end{proof}

\begin{proof}[Proof of Lemma~\ref{1a}]
  \lb{pg1a} Due to the continuity assumption every function
  $f\in\dom(H(\cL, \ell))$ satisfying
  \begin{equation*}
    (\tr f)\upharpoonright_{\partial\Gamma_+}=0, 
  \end{equation*}	
  and 
  \begin{equation*}
    f(a_e)=f(b_e),\, \text{for all}\ e\in \cE_0          
  \end{equation*}
  has zero Dirichlet trace: $\gamma_D f=0$. 
\end{proof}


	

\section{Resolvent estimates and the spectral convergence }
\label{aux}

As mentioned in Section~\ref{sec:main_results}, in order to prove
spectral convergence, Theorem~\ref{thm:spectral_convergence}, we will
require some technical estimates listed in Theorem~\ref{ss1}.  Before
we formally prove Theorems~\ref{ss1} and \ref{y1}, we compare these
estimates with standard functional-analytic results.

Part \emph{(i)} of Theorem~\ref{ss1} gives a bound on the resolvent of a
quantum graph operator.  A well-known bound on the resolvent of a
general self-adjoint operator $H$ on a Hilbert space $\cH$ gives
\begin{equation}
  \label{eq:gen_resolvent_bound}
  \|(H-zI)^{-1}\|_{\cB(\cH)} \leq \frac{1}{|\Im z|},\qquad \Im z \neq 0.
\end{equation}
We immediately get, for any $\Gamma(\ell)$,
\begin{equation}
  \label{eq:gen_resolvent_bound_i}
  \|R(\cL,\ell, \bfi)\|_{\cB(L^2(\Gamma(\ell)))}\leq 1.
\end{equation}
We stress that this bound is weaker than part \emph{(i)} of
Theorem~\ref{ss1}, which bounds $R(\cL,\ell, \bfi)$ as an operator
from $L^2$ to $L^\infty$.  

On the other hand, part {\it(iii)} of Theorem~\ref{ss1} is reminiscent
of the following standard Sobolev-type inequalities that hold for all
edges $e$,
\begin{align}
  &\|f_e\|_{L^{\infty}(e)} 
    \leq {\ell_e^{-1/2}}{\|f_e\|_{L^2(e)}}
    +{\ell_e^{1/2}\|f'_e\|_{L^2(e)}},
    \label{g5}\\
  &\|f_e'\|_{L^{2}(e)} 
    \leq {\ell_e^{-1}}{\|f_e\|_{L^2(e)}}
    +{\ell_e \|f_e''\|_{L^2(e)}},
    \label{g5.1}\\
  &\|f'_e\|_{L^{\infty}(e)}
    \leq {\ell_e^{-1/2}}{\|f'_e\|_{L^2(e)}}
    +{\ell_e^{1/2}\|f''_e\|_{L^2(e)}},
    \label{g5.2}
\end{align}
(cf., e.g, \cite[Theorem 4.2.4 and Corollary  4.2.7 part 1.]{Bu}). 
However, in a situation when some edge lengths $\ell_e \to 0$, uniform
bound \eqref{ss4} is a substantially stronger statement.

\begin{proof}[Proof of Theorem \ref{ss1}]\lb{pgss1} By the resolvent
  identity it suffices to verify equivalency of the statements for the
  free resolvent, i.e.\ we may assume that $q^{\ell}\equiv0$ for all
  $\ell$. Indeed, denoting the free resolvent by
  $R_0(\cL,\ell, \bfi):=(H_0(\cL, \ell)-\bfi)^{-1}$ one has
  \begin{equation}
    \label{eq:resolvent_identity}
    R(\cL,\ell, \bfi)
    =R_0(\cL,\ell, \bfi)-R_0(\cL,\ell, \bfi)q^{\ell}R(\cL,\ell, \bfi).
  \end{equation}
  Next, we recall that by assumptions 
  \begin{equation}
    \|q^{\ell}\|_{ L^{\infty}(\Gamma(\ell);\bbR)}\leq c,
  \end{equation}
  for some $c>0$ and all $\ell$ sufficiently close to
  $\wti\ell$. Combining this bound with
  \eqref{eq:gen_resolvent_bound_i} and \eqref{eq:resolvent_identity}
  one infers that parts $(i)$ and $(ii)$ hold if and only if they hold
  with $q^{\ell}\equiv0$. In addition, since
  $\dom(H(\cL,\ell))=\dom(H_0(\cL,\ell))$, part $(iii)$ holds if and
  only it holds with $q^{\ell}\equiv0$.
	
	{\it(i)}$\implies${\it(ii)}. For arbitrary $e\in\cE_0$ and  $v\in L^2(\Gamma)$,
	\begin{align}
	&\|\chi_eR_0(\cL, \ell, \bfi)v\|_{L^2(\Gamma(\ell))}\leq \|\chi_{e}\|_{L^{2}(\Gamma(\ell))}\|R_0(\cL, \ell, \bfi)v\|_{L^{\infty}(\Gamma(\ell))}\label{g3.1}\\
	&\leq \ell_e^{1/2} \|R_0(\cL, \ell, \bfi)\|_{\cB{(L^2(\Gamma(\ell)), L^{\infty}(\Gamma(\ell)))}} \|v\|_{L^2(\Gamma(\ell))}.
	\end{align}
	Combining this with {\it(i)} we infer {\it (ii)}.
	
	{\it(ii)}$\implies${\it(iii)}. Let $e\in\cE_0$. For $f\in\dom(H_0(\cL,\ell))$ put $-f''-\bfi f=v$, then by {\it (ii)},
	\begin{align}\label{ss10}
	\|f_e\|_{L^2(e)}\lesssim \sqrt{\ell_e} \|v\|_{L^2(\Gamma)}=\sqrt{\ell_e} \|f''+\bfi f\|_{L^2(\Gamma(\ell))}.
	\end{align}
Combining  \eqref{ss10}, \eqref{g5}, \eqref{g5.1} we obtain that for every $e\in\cE$,
	\begin{align}
	\|f_e\|_{L^{\infty}(e)}&\leq2{\ell_e^{-1/2}}{\|f_e\|_{L^2(e)}}+{\ell_e^{3/2} \|f_e''\|_{L^2(e)}}\label{ss11}\\
	&\lesssim \|f''+\bfi f\|_{L^2(\Gamma(\ell))}+{\ell_e^{3/2} \|f_e''\|_{L^2(e)}},\\
	&\lesssim c(\ell) \left(\|f\|_{L^{2}(\Gamma(\ell))}+\|f''\|_{L^{2}(\Gamma(\ell))}\right),\label{ss12}
	\end{align}
	where $c(\ell)=\cO(1)$ as $\ell\rightarrow\wti\ell$. Note that we had to use \eqref{ss10} because $\ell_e\rightarrow 0$ for $e\in\cE_0$.
	
	{\it(iii)}$\implies${\it(i)}  Let $f\in\dom(H_0(\cL,\ell))$ and let $-f''-\bfi f=v$. Then by {\it (iii)}, 
	\begin{align}
          &\|R_0(\cL,\ell, \bfi)v\|^2_{L^{\infty}(\Gamma(\ell))}
            =\|f\|^2_{L^{\infty}(\Gamma(\ell))}\\
          &\quad\lesssim c \left(\|f\|^2_{L^{2}(\Gamma(\ell))}
            +\|v+\bfi f\|^2_{L^{2}(\Gamma(\ell))}\right)\label{ss20}\\
          &\quad\lesssim c
            \left(\|R_0(\cL,\ell,\bfi)v\|^2_{L^{2}(\Gamma(\ell))}
            +\|v\|^2_{L^{2}(\Gamma(\ell))}\right)\label{ss20.1}\\
          &\quad\lesssim c \|v\|^2_{L^{2}(\Gamma(\ell))},\label{ss20.2}
	\end{align}
	where in the last step we used \eqref{eq:gen_resolvent_bound_i}.
	This proves {\it(i)}.
	
	Next we prove \eqref{qqq1} and \eqref{b8}. To this end,  let $f\in\dom(H_0(\cL,\ell))$ and let $-f''-\bfi f=v$. Then
	using \eqref{aa14} and the Cauchy--Schwarz inequality, we  obtain
	\begin{align}
	\|f'\|^2_{L^2(\Gamma(\ell))}&\leq \big| \langle f, f''\rangle_{L^2(\Gamma(\ell))}\big|+ \big|\langle P_R \gamma_D^{\ell} f, Q P_R \gamma_D^{\ell} f\rangle_{L^2(\partial\Gamma)}\big|\\
	&\leq \|f\|_{L^2(\Gamma(\ell))}\|f''\|_{L^2(\Gamma(\ell))}+\|Q\| \|\gamma_D^{\ell} f\|_{L^2(\partial\Gamma)}^2\\
	& \leq \|f\|^2_{L^2(\Gamma(\ell))}+\|f''\|^2_{L^2(\Gamma(\ell))}+\|Q\| \|\gamma_D^{\ell} f\|_{L^2(\partial\Gamma)}^2\label{ss14}.
	\end{align}
	Employing \eqref{ss4} we estimate the third term in \eqref{ss14} and infer \eqref{qqq1}.
	
	Then, one has
	\begin{align}
	\|R_0(\cL,\ell, \bfi)v\|_{\hatt H^2(\Gamma(\ell))}&=\|f\|_{\hatt H^2(\Gamma(\ell))}^2=\|f\|_{L^2(\Gamma(\ell))}^2+\|f'\|_{L^2(\Gamma(\ell))}^2+\|f''\|_{L^2(\Gamma(\ell))}^2\no\\
	&\lesssim c(\ell)\left(\|f\|_{L^2(\Gamma(\ell))}^2+\|f''\|_{L^2(\Gamma(\ell))}^2\right)\lesssim c(\ell) \|v\|^2_{L^{2}(\Gamma(\ell))},\no
	\end{align}
	where $c(\ell)\underset{\ell\rightarrow\wti\ell}{=}\cO(1)$ and in the last step we proceeded as in \eqref{ss20}-\eqref{ss20.2}. 
\end{proof}

In the proof of Theorem \ref{y1} we will use the following geometric fact.
\begin{proposition}\label{nw10}
  Suppose that $A$ and $B$ are closed linear subspaces of a Hilbert
  space $X$, and that at least one of them is finite dimensional. Let
  $\{b_n\}_{n=1}^{\infty}\subset B$ be such that
  $\dist(b_n, A)\rightarrow0$ as $n\rightarrow\infty$. Then
  $\dist(b_n, A\cap B)\rightarrow0$ as $n\rightarrow\infty$.
\end{proposition}

As the following counterexample\footnote[1]{Due to Th.~Schlumprecht}
shows, the proposition may not hold if both $A$ and $B$ are infinite
dimensional. In the sequence space $X=\ell^2(\bbN)$ we consider
infinite dimensional  subspaces 
\begin{align}
  A &= \{(x_1,x_2,x_2,x_3,x_3,\dots)\in\ell^2(\bbN): x_k\in\bbC\}\\
  B &= \{(x_1,x_1,x_2,x_2,x_3,x_3,\dots)\in\ell^2(\bbN): x_k\in\bbC\},
\end{align}
and let
\begin{align}
  a_n &= (1,\ 1-\tfrac1n,\ 1-\tfrac1n,\ 
        1-\tfrac2n,\ 1-\tfrac2n,
        \ldots,\ \tfrac1n,\ 0,\ 0, \dots)\in A,\\
  b_n &= (1,\ 1,\phantom{-\tfrac1n}\ \ 1-\tfrac1n,\ 1-\tfrac1n,\ 
        1-\tfrac2n, \ldots,\ \tfrac1n,\ \tfrac1n,\ 0,\ldots)\in B,
\end{align}
for $n=1,2,\dots$. Then
$\dist(b_n,A)\le\|b_n-a_n\|=n^{-1/2}\rightarrow0$ while
$A\cap B=\{0\}$ and $\dist(b_n,A\cap B)=\|b_n\|\rightarrow+\infty$ as
$n\rightarrow\infty$.

\begin{proof}[Proof of Proposition~\ref{nw10}]
  Let $P$ denote the orthogonal projection onto $(A\cap B)^\bot$. We
  want to show that $\|Pb_n\|=\dist(b_n, A\cap B)\rightarrow0$ as
  $n\rightarrow\infty$.  Note that $Pb_n \in B$.

  Consider the orthogonal decomposition 
  \begin{equation}
    \label{eq:orth_decomp_X}
    X = \Big( (A\cap B) 
    \oplus \left(A\cap(A\cap B)^\bot\right) \Big) \oplus A^\bot,
  \end{equation}
  and split $b_n$ accordingly, $b_n=x_n+y_n+z_n$.  Applying $P$, we
  see that $Pb_n=y_n+Pz_n$.  We know that
  $\|z_n\| = \dist(b_n, A) \to 0$, therefore $Pz_n \rightarrow0$ and
  we conclude that either both sequences $(Pb_n)$ and $(y_n)$ converge
  to zero (and then the proof is finished), or else, may be by passing
  to a subsequence, they both are separated away from zero. Let us
  suppose that the latter holds. Then equality
  $\|y_n\|^{-1}Pb_n=\|y_n\|^{-1}y_n+\|y_n\|^{-1}Pz_n$ shows that the
  following two sequences,
  \begin{equation*}
    \big(\|y_n\|^{-1}Pb_n\big)\subset B\cap(A\cap B)^\bot
    \quad\text{ and }\quad
    \big(\|y_n\|^{-1}y_n\big)\subset  A\cap(A\cap B)^\bot,
  \end{equation*}
  are bounded. Since at least one of the subspaces $A$ or $B$ is
  finite dimensional, passing to a subsequence, we may conclude that
  at least one of the two sequences converges. Then by
  $\|y_n\|^{-1}Pz_n\rightarrow0$ both sequences must converge, and
  their common limit must be zero as it belongs to $A\cap B$ and
  $(A\cap B)^\bot$. Since $\|y_n\|^{-1}y_n$ is of unit length, the
  contradiction completes the proof.
\end{proof}

\begin{proof}[Proof of Theorem \ref{y1}]\lb{pgy1}	
  Due to the resolvent identity,
  equation~\eqref{eq:resolvent_identity}, it is enough to prove the
  statement for the free Laplacian. That is, we focus on the case of
  zero potential.
	
  Seeking a contradiction we assume that condition({\it iii}) from
  Theorem \ref{ss1} does not hold and obtain sequences
  $\{\ell_n\}_{n=1}^{\infty}\subset \bbR^{|\cE|}_{>0}$ and
  $\{\varphi_n\}_{n=1}^{\infty}\subset \dom( H(\cL,\ell_n))$ such that
  \begin{align}
    &\ell_n \to \wti\ell,\label{8.6}\\
    &\|\varphi_n\|_{L^{\infty}(\Gamma(\ell_n))}=1,
      \quad n\in\bbN,\label{y6}\\
    &\|\varphi_n\|_{L^{2}(\Gamma(\ell_n))}
      +\|\varphi_n''\|_{L^{2}(\Gamma(\ell_n))}\rightarrow0,
      \quad n\rightarrow\infty\label{y5}
  \end{align}
  From equation (\ref{aa14}) one has
  \begin{align}
    \|\varphi_n'\|^2_{L^2(\Gamma(\ell_n))}=\langle \varphi_n,
    \varphi_n''\rangle_{L^2(\Gamma(\ell_n))}
    - \langle P_R \gamma_D\varphi_n, QP_R \gamma_D
    \varphi_n\rangle_{L^2(\partial\Gamma)}.
  \end{align}
  Thus, using \eqref{y6}, \eqref{y5} we get
  \begin{equation}\label{pq21}
    \|\varphi_n'\|_{L^2(\Gamma(\ell_n))} \underset{n\rightarrow \infty}{=}\cO(1).
  \end{equation}
  Using this, for each $e\in\cE_0$ one obtains
  \begin{equation}\label{pq20}
    |\varphi_n(a_e)-\varphi_n(b_e)|
    =\left|\int_{e}\partial_{\nu}\varphi_n\right| 
    \lesssim\sqrt{\ell_{n,e}}\|\varphi_n'\|_{L^2(\Gamma(\ell_n))}
    \rightarrow 0,\quad n\rightarrow \infty.
  \end{equation}
  Similarly, by \eqref{y5}, for each $e\in\cE_0$ one has
  \begin{equation}\label{nw1}
    | \varphi'_n(a_e)-\varphi'_n(b_e)|
    =\left|\int_{e}\varphi''_n\right| 
    \lesssim\sqrt{\ell_{n,e}}\|\varphi_n''\|_{L^2(\Gamma(\ell_n))} 
    \rightarrow 0,\quad n\rightarrow \infty.
  \end{equation}
  That is,
  \begin{equation}\label{nw7}
    \dist\big(\tr \varphi_n, D_0\oplus N_0\big) \rightarrow 0,\ n\rightarrow\infty.
  \end{equation}
  Next, using \eqref{y5} and the standard Sobolev inequalities on
  $\Gamma_+$, cf. \eqref{g5}--\eqref{g5.2}, we obtain
  \begin{align}
    &\|\varphi_n\|_{L^{\infty}(\Gamma_+(\ell_n))}\lesssim {\|\varphi_n\|_{L^2(\Gamma_+(\ell_n))}}+\|\varphi''_n\|_{L^2(\Gamma_+(\ell_n))}\underset{n\rightarrow \infty}{=}o(1),\label{y16}\\
    &\|\varphi_n'\|_{L^{\infty}(\Gamma_+(\ell_n))}\lesssim {\|\varphi_n\|_{L^2(\Gamma_+(\ell_n))}}+\|\varphi''_n\|_{L^2(\Gamma_+(\ell_n))}\underset{n\rightarrow \infty}{=}o(1).\label{y17}
  \end{align}
  In particular,
  \begin{equation}\label{pq30}
    \lim_{n\rightarrow\infty}\|\varphi_n\upharpoonright_{\partial\Gamma_+}\|_{L^{\infty}(\partial\Gamma_+)}=0,\ \lim\limits_{n\rightarrow\infty}\|\varphi_n'\upharpoonright_{\partial\Gamma_+}\|_{L^{\infty}(\partial\Gamma_+)}=0.
  \end{equation}
  Moreover, one has 
  \begin{equation}\label{8.18}
    \liminf\limits_{n\rightarrow\infty}\|\varphi_n\upharpoonright_{\partial\Gamma_0}\|_{L^{\infty}(\partial\Gamma_0)}>0.
  \end{equation}
  Indeed, assuming the contrary and passing to a subsequence if
  necessary, one gets that for any $e\in\cE_0$ and arbitrary
  $x\in e$
  \begin{equation}
    |\varphi_n(x)|
    \leq |\varphi_n(a_e)|+\left|\int_{e}\partial_{\nu}\varphi_n\right| 
    \leq |\varphi_n(a_e)|
    +\sqrt{\ell_{n,e}} \|\varphi_n'\|_{L^2(\Gamma(\ell_n))} 
    \to 0, \quad n\to\infty,
  \end{equation}
  contradicting \eqref{y6}. 
  
  Next, using  \eqref{nw7} and \eqref{pq30} we obtain
  \begin{equation}\label{nw3}
    \dist \big( \tr \varphi_n, (D_0\oplus N_0)\cap \ker(\,\dP_+) \big) \rightarrow 0,\ n\rightarrow\infty.
  \end{equation}
  Combining this with $\tr\varphi_n \in \cL$ and Proposition
  \ref{nw10}, we obtain that
  \begin{equation}
    \dist \big( (\phi_1^n,\phi_2^n), \cL\cap (D_0\oplus N_0)\cap \ker(\,\dP_+) \big) \rightarrow 0,\ n\rightarrow\infty.
  \end{equation}
  Interpreting Condition~\ref{newhyp} as
  $\cL\cap (D_0\oplus N_0)\cap \ker(\,\dP_+) \subset \{0\}\oplus
  L^2(\partial\Gamma)$, one has
  \begin{align}
    \|\gamma_D \varphi_n\|_{L^{2}(\partial\Gamma)}
    &= \dist \big( \tr\varphi_n, \{0\}\oplus L^2(\partial\Gamma) \big)\\
    &\leq  \dist \big( \tr\varphi_n, 
      \cL\cap (D_0\oplus N_0)\cap \ker(\,\dP_+) \big) \rightarrow 0,\ n\rightarrow\infty.
  \end{align}
  which contradicts \eqref{8.18}. 
	
  To prove the last statement assume that $P_R=0$. Then by Lemma
  \ref{y21} there exists a nonzero function $f$ constant on each edge
  satisfying the boundary conditions and such that
  supp$(f)\subset \Gamma_{0}$. Since $f''=0$ and
  $\| f\|_{L^2(\Gamma(\ell))}^2\rightarrow0$ as
  $\ell\rightarrow\wti \ell$, the inequality \eqref{ss4} does not
  hold.	
\end{proof}

As was pointed out in Introduction, our method of proving spectral
convergence relies upon a technique developed by P.\ Exner and O.\
Post \cite{EP, P06, P11, P12}.

\begin{definition}\label{d1}
  For each $t\in\bbR^n, n\in\bbN,$ let $H_{t}$ be a self-adjoint
  operator acting in the Hilbert space $\cH_{t}$. Then $H_{t}$ is said
  to converge in the \emph{generalized norm resolvent sense} to
  $H_{\wti{t}}$, as $t\rightarrow\wti{t}$ if for each $t\in\bbR^n$
  there exists a bounded linear operator
  $\cJ_{t}\in\cB(\cH_{\wti{t}},\cH_{t})$ such that	
  \begin{align}
    \cJ_{t}^*\cJ_{t}=I_{\cH_{\wti t}} &\text{\ for all\ }t\in\bbR^n,\label{c2.2}\\ 
    \|(I_{\cH_{t}}-\cJ_{t}\cJ_{t}^*)(H_{t}-&zI_{\cH_t})^{-1}\|_{\cB(\cH_{t})}\underset{t\rightarrow\wti t}{=}o(1), \label{c2}\\
    \|\cJ_{t}(H_{\wti{t}}-zI_{\cH_{\wti t}})^{-1}-(H_{t}-z&I_{\cH_t})^{-1}\cJ_{t}\|_{\cB(\cH_{\wti{t}},\cH_{t})}\underset{t\rightarrow\wti t}{=}o(1)\label{c3} ,
  \end{align}
  for each $z\in \C$, $\Im z\not=0$.
  In this case we write $H_{t} \xrightarrow[\text{}]{\text{gnr}}
  H_{\wti{t}}$, as $t\rightarrow \wti{t}$. 
\end{definition}

Assuming conditions \emph{(i)}-\emph{(iii)} of Theorem~\ref{ss1} we
focus on showing that
\begin{equation}\label{gnr}
  H(\cL, \ell) \xrightarrow[\text{}]{\text{gnr}} H(\wti\cL, \wti
  \ell),
  \quad \ell\rightarrow \wti\ell.
\end{equation}

As a first step, we show that, in the abstract setting, the
generalized norm resolvent convergence is preserved under bounded
perturbations.

\begin{theorem}
  \label{A1}
  Let $H^0_t$, $H^0_{\wti{t}}$ satisfy Definition \ref{d1}. Let
  $A_t\in\cB(\cH_t)$ be a family of self-adjoint, bounded operators
  satisfying the relations
  \begin{equation}\label{ap1}
    \|\cJ_{t}A_{\wti{t}}-A_{t}\cJ_{t}\|_{\cB(\cH_{\wti{t}},\cH_{t})}
    \underset{t\rightarrow\wti t}{=}o(1)
    \quad  \text{and }\quad
    \|A_{t}\|_{\cB(\cH_{t})}\underset{t\rightarrow\wti t}{=}\cO(1).
  \end{equation} 
  Then equations~\eqref{c2} and \eqref{c3} hold with
  $H_t = H^0_t + A_t$.
\end{theorem}

\begin{proof}
The proof relies on the resolvent identity
\begin{align}
R(t) &= R_0(t)-R(t)A_tR_0(t)\label{ap4}\\
     &=R_0(t) + R_0(t)A_tR(t), \ t\in\bbR^n. \label{ap41}
\end{align}
where
\begin{equation}\no
  R(t):=(H^0_{t}+A_{t}-zI_{\cH_t})^{-1},
  \qquad  R_0(t):=(H^0_{t}-zI_{\cH_t})^{-1}, 
  \quad \Im z\not=0.
\end{equation}
In order to verify \eqref{c2} for $H_t = H^0_t + A_t$, we combine
\eqref{c2} (for $R_0(t)$) and \eqref{ap41} and obtain
\begin{align*}
\|(I_{\cH_{t}}-\cJ_{t}\cJ_{t}^*)R(t)\|_{\cB(\cH_{t})}&\leq \|(I_{\cH_{t}}-\cJ_{t}\cJ_{t}^*)R_0(t)\|_{\cB(\cH_{t})}
+\|(I_{\cH_{t}}-\cJ_{t}\cJ_{t}^*)R_0(t)A_tR(t)\|_{\cB(\cH_{t})} \\
&\underset{t\rightarrow\wti t}{=}o(1)\left(1+\|A_tR(t)\|_{\cB(\cH_{t})}\right)\underset{t\rightarrow\wti t}{=}o(1), \no
\end{align*}
where we used the second equality in \eqref{ap1}, and the general
resolvent bound \eqref{eq:gen_resolvent_bound}.

The identity
\begin{multline}
  \label{ap6}
  \left(\cJ_{t}R(\wti t) - R(t)\cJ_{t}\right)
  \left(I_{\cH_{\wti t}}+A_{\wti t}R_0(\wti t)\right) \\ \nonumber
  = \left(I_{\cH_{t}}-R(t) A_t\right) \left(\cJ_t R_0(\wti t) - R_0(t) \cJ_t\right) 
  + R(t)\left(A_t \cJ_{t} - \cJ_{t}A_{\wti t}\right)R_0(\wti t)
\end{multline}
may be verified by substituting \eqref{ap4} for $R(\wti t)$ and $R(t)$
on the left-hand side and expanding.  Using \eqref{c3}, \eqref{ap1}
and \eqref{eq:gen_resolvent_bound}, we arrive at
\begin{equation}\label{ap7}
  \left\|\left(\cJ_{t}\wti R-R(t)\cJ_{t}\right)
    \left(I_{\cH_{\wti t}}+A_{\wti t}R_0(\wti t)\right)
    \right\|_{\cB(\cH_{\wti t},\cH_{t})}
  \underset{t\rightarrow\wti t}{=}o(1).
\end{equation}
Moreover, due to the identity
\begin{equation*}
  I_{\cH_{\wti t}}+A_{\wti t}R_0(\wti t) = 
  \left(H_{\wti t}-z I_{H_{\wti t}}+A_{\wti t}\right)R_0(\wti t),
\end{equation*}
the operator $I_{\cH_{\wti t}}+A_{\wti t}\wti R_0$ is boundedly
invertible on $\Im z \neq 0$. Thus \eqref{ap7} implies \eqref{c3} for
$H_t=H^0_t+A_t$.
\end{proof}

In the following theorem we establish a version of \eqref{c2.2} and \eqref{c2} in the context of graphs with vanishing edges. Let us recall definition of $\cJ_{\ell}$ from \eqref{d5}.

\begin{theorem}\label{uu77}Assume conditions {\it (i)-(iii)} of
  Theorem~\ref{ss1} hold.  Then
  \begin{equation}
    \cJ_{\ell}^*\cJ_{\ell}=I_{L^2(\Gamma(\wti\ell))},\ \ell\in\bbR^{|\cE|}_{>0},\label{d3}\\ 
  \end{equation}
  and
  \begin{equation}
    \big\|(I_{L^2(\Gamma(\cG;\ell))}-\cJ_{\ell}\cJ_{\ell}^*)R(\cL,\ell, z)\big\|_{\cB(L^2(\Gamma(\cG;\ell)))}\underset{\ell\rightarrow\wti\ell}{=} o(1), \label{d4}
  \end{equation}
  for each $z\in\bbC$, $\Im z\not=0.$
\end{theorem}

\begin{proof}
	Using change of variables, one obtains
	\begin{align}
	&\cJ_{\ell}^*\in\cB\big(L^2(\Gamma(\ell)),L^2(\Gamma(\wti\ell))\big),\label{d6}\\
	&(\cJ^*_{\ell}f)_e(x):=\sum_{e\in\cE_+} \chi_{e}(x)
          \sqrt{\frac{\ell_j}{\tilde{\ell}_j}}
          f_e\left(\frac{x\ell_j}{\tilde{\ell}_j}\right),
          \quad f\in L^2(\Gamma(\ell)),\ x\in\Gamma(\wti\ell).\label{d7}
	\end{align}
	A direct computation shows that \eqref{d5} and \eqref{d7} yield \eqref{d3}.  
	Moreover, one has
	\begin{equation}
          \cJ_{\ell}\cJ_{\ell}^*\in\cB(L^2(\Gamma(\ell))), 
          \quad \cJ_{\ell}\cJ_{\ell}^*f=\chi_{\Gamma_+(\ell)} f,
          \quad f\in L^2(\Gamma(\ell)), \label{d8}\\
	\end{equation}
	where $\chi_{\Gamma_+(\ell)}$ denotes the characteristic  function of ${\Gamma_+(\ell)}$.
	
	By Theorem \ref{ss1}{\it (ii)} one has
	\begin{align}
	\big\|(I_{L^2(\Gamma(\ell))}-\cJ_{\ell}\cJ_{\ell}^*)R(\cL,
          \ell, z)\big\|_{L^2(\Gamma(\ell))}
          &=\big\| \sum_{e\in\cE_0}\chi_e R(\cL, \ell, z) \big\|_{L^2(\Gamma(\ell))}\label{g1}\\
	\quad\leq \sum_{e\in\cE_0}\left\|\chi_e R(\cL, \ell, z)\right\|_{L^2(\Gamma(\ell))}&\underset{\ell\rightarrow\wti\ell}{=}\sum_{e\in\cE_0}\cO(\ell_e^{1/2})\underset{\ell\rightarrow\wti\ell}{=}o(1),
	\end{align}
	as asserted.
\end{proof}

In the following theorem we establish a version of \eqref{c3} in the
context of graphs with vanishing edges. Together with
Theorem~\ref{uu77} this will conclude the proof of Theorem~\ref{cc1}.

\begin{theorem}\label{u7}
  Assume conditions {\it(i)-(iii)} of Theorem~\ref{ss1}, and recall
  the operator $H(\wti\cL, \wti\ell)$ from Theorem~\ref{a30.1}. Then
  \begin{equation}
    \label{u8}
    \Big\|\cJ_{\ell}R(\wti\cL,\wti\ell,z)-R(\cL,\ell,z)\cJ_{\ell}\Big\|_{\cB\left(L^2(\Gamma(\wti\ell)), L^2(\Gamma(\ell))\right)}\underset{\ell\rightarrow\wti\ell}{=} o\left(1\right),
  \end{equation}
  for each $z\in\bbC$, $\Im z\not=0$.
\end{theorem}

\begin{proof}
  We split the proof into several natural steps. In the first step we
  prove \eqref{u8} in the situation when the non-vanishing edges are
  being fixed while the vanishing edges tend to zero. This is the most
  challenging part of the proof. In the second step we deal with
  \eqref{u8}when the vanishing edges are absent, while the
  non-vanishing edges rescale non-singularly. Finally, in the third
  step we put everything together, and obtain \eqref{u8} as asserted.
	
  Note that by Theorem \ref{A1} we may assume that $q^{\ell}\equiv 0$
  for all $\ell$.
	
  {\bf Step 1.} Let us denote
  $\ell=(\ell_+, \ell_0)$, $\hatt\ell:= (\ell_+, 0)$. Then the scaling
  operator acting from $\Gamma(\hatt\ell)$ to $\Gamma(\ell)$ is given
  by
  $\bbJ_{\ell,\hatt\ell} \in
  \cB\big(L^2(\Gamma(\wti\ell)),L^2(\Gamma(\hatt\ell))\big)$,
  \begin{equation}
    \label{eq:bbJ_def}
    (\bbJ_{\ell,\hatt\ell} f)(x)
    = \sum_{e\in\cE_+}
    \chi_{e}(x) f\left(x\right),
    \quad x\in\Gamma(\ell),
  \end{equation}

  The goal of this step is to prove a version of \eqref{u8} with
  respect to this scaling operator. Namely, we will prove that
  \begin{align}
    \begin{split}\label{u8new}
      \Big\|\bbJ_{\ell,\hatt\ell}R(\wti\cL,\hatt\ell,z)-R(\cL,\ell,z)\bbJ_{\ell,\hatt\ell}\Big\|_{\cB\left(L^2(\Gamma(\hatt\ell)), L^2(\Gamma(\ell))\right)}\underset{\ell_0\rightarrow 0}{=} o\left(1\right),
    \end{split}
  \end{align}
  holds uniformly in $\ell_+$ satisfying
  \begin{equation}\label{az5}
    \frac{|\wti \ell|}{2}\leq |\ell_+|\leq |\wti\ell|.
  \end{equation}
  
  It suffices to prove that the inequality 
  \begin{equation}
    \label{u9}
      \left|\left\langle f, \big(\bbJ_{\ell,\hatt\ell}R(\wti\cL,\hatt\ell,z)-R(\cL,\ell,z)\bbJ_{\ell,\hatt\ell} \big)g\right\rangle_{L^2(\Gamma(\ell))}\right|\leq c(\ell) \|f\|_{L^2(\Gamma(\ell))} \|g\|_{L^2(\Gamma(\hatt\ell))},
  \end{equation}
  holds for arbitrary $f\in L^2(\Gamma(\ell))$ and $g\in L^2(\Gamma(\hatt\ell))$, with 
  \begin{equation}\label{az6}
    \sup\limits_{\ell_+:\  \frac{|\wti \ell|}{2}\leq |\ell_+|\leq |\wti\ell|}c(\ell)=o\left(1\right)\  \text{as\ }
    \ell_0\rightarrow 0.
  \end{equation}
  Let us denote
  \begin{equation}\label{u10}
    u:=R(\cL,\ell,\overline{z})f \text{\ and\ }v:=R(\wti\cL,\hatt\ell,z)g.
  \end{equation}
  Rewriting  the left-hand side of \eqref{u9} we obtain
  \begin{align}
    &\Big\langle f, \bbJ_{\ell,\hatt\ell}\big(H(\wti\cL, \hatt\ell)-z\big)^{-1}g\Big\rangle_{L^2(\Gamma(\ell))}-\Big\langle \left(H(\cL, \ell)-\overline{z} \right)^{-1}f, \bbJ_{\ell,\hatt\ell} g\Big\rangle_{L^2(\Gamma(\ell))}\label{pr1}\\
    &\quad=\Big\langle \big(H(\cL, \ell)-\overline{z}\big)u, \bbJ_{\ell,\hatt\ell}v\Big\rangle_{L^2(\Gamma(\ell))}-\Big\langle u, \bbJ_{\ell,\hatt\ell} \big(H(\wti\cL, \hatt\ell)-z\big) v\Big\rangle_{L^2(\Gamma(\ell))}\\
    &\quad=\Big\langle H(\cL, \ell)u, \bbJ_{\ell,\hatt\ell}v\Big\rangle_{L^2(\Gamma(\ell))}-\Big\langle \overline{z}u, \bbJ_{\ell,\hatt\ell}v\Big\rangle_{L^2(\Gamma(\ell))}\\
    &\hspace{3.21cm}- \Big\langle u, \bbJ_{\ell,\hatt\ell} H(\wti\cL, \hatt\ell) v\Big\rangle_{L^2(\Gamma(\ell))} +\Big\langle u,z \bbJ_{\ell,\hatt\ell} v\Big\rangle_{L^2(\Gamma(\ell))}   \\
    &\quad=\Big\langle H(\cL, \ell)u, \bbJ_{\ell,\hatt\ell}v\Big\rangle_{L^2(\Gamma(\ell))}-\Big\langle u, \bbJ_{\ell,\hatt\ell} H(\wti\cL, \hatt\ell) v\Big\rangle_{L^2(\Gamma(\ell))}.\label{pr2}
  \end{align}
  Henceforth, our objective is to show that 
  \begin{equation}
    \label{u22new}
    \left|\left\langle H(\cL, \ell)u,
        \bbJ_{\ell,\hatt\ell}v\right\rangle_{L^2(\Gamma(\ell))}
      -\big\langle u, \bbJ_{\ell,\hatt\ell} H(\wti\cL, \hatt\ell)
      v\big\rangle_{L^2(\Gamma(\ell))}\right| 
    = o\left(1\right) \|f\|_{L^2(\Gamma(\ell))} \|g\|_{L^2(\Gamma(\hatt\ell))},
  \end{equation}
  as $\ell_0\rightarrow0$, uniformly in $\ell_+$ satisfying \eqref{az5}.
  
  Denoting the left-hand side by $Z$ and integrating by parts one obtains
  \begin{align}
    Z &:= \big\langle H(\cL, \ell)u, \bbJ_{\ell,\hatt\ell}v\big\rangle_{L^2(\Gamma(\ell))}-\big\langle u, \bbJ_{\ell,\hatt\ell} H(\wti\cL, \hatt\ell) v\big\rangle_{L^2(\Gamma(\ell))}\\
      &=\int_{\Gamma_+(\ell)} \overline{u''}\bbJ_{\ell,\hatt\ell}v-\overline{u}\bbJ_{\ell,\hatt\ell}v''=\int_{\Gamma_+(\ell)} \overline{u''}v-\overline{u}v''=\int_{\partial\Gamma_+} \overline{\partial_{\nu}u} v-\overline{u}\partial_{\nu}v,\label{vv26}
  \end{align}
  where we used
  \begin{equation}\no
    (\bbJ_{\ell,\hatt\ell}f)(x)=\chi_{\Gamma_{+}(\ell)}f(x),\ x\in\Gamma_{+}(\ell)
  \end{equation}
  due to the fact that the fact that $\ell_+$ is fixed.
  
  By Theorem \ref{a30.1},
  $\tr v \in \wti\cL = \dP_+ \big(\cL \cap (D_0 \oplus N_0) \big)$.
  Let $G : \wti\cL \to \cL \cap (D_0 \oplus N_0)$ be any
  finite-dimensional linear operator\footnote{That is, $G$ is a
    ``generalized inverse'' of $\dP_+$.  It always exist but may no be
    unique; the choice of $G$ with the least norm is the Moore-Penrose
    pseudoinverse.} such that $\dP_+ G \phi = \phi$ for any
  $\phi\in\cL$. We let
  \begin{equation}\no
    w=(w_1, w_2) = G \tr v \in\cL\cap\big( D_0\oplus N_0
    \big)\subset\,\dL^2(\partial\Gamma);
    \label{vv34}
  \end{equation}
  it satisfies
  \begin{align}
    & \dP_+ w = \tr v,\label{vv32}\\
    &\|\dP_0 w\|_{\dL^2(\partial\Gamma_0)}
      \leq \|w\|_{\dL^2(\partial\Gamma)}
      \lesssim \|\tr v\|_{\dL^2(\partial\Gamma_+)},\label{vv33}
  \end{align}
  the latter because $G$, as any finite-dimensional linear operator,
  is bounded.

  Using \eqref{vv32} we rewrite the last integral in \eqref{vv26},
  \begin{align}
    Z = \int_{\partial\Gamma_+} \overline{\partial_{\nu}u} v-\overline{u}\partial_{\nu}v=\omega_{\Gamma}(\dP_{+}\tr^{\ell}u,\, \dP_{+}w).\label{vv40new}
  \end{align}
  Since $\omega_{\Gamma}(\tr^{\ell}u,w)=0$, equation \eqref{aa19} yields
  \begin{equation}\label{vv35new}
    Z = \omega_{\Gamma}(\dP_{+}\tr^{\ell}u,\, \dP_{+}w)=\omega_{\Gamma}(\dP_{0}\tr^{\ell}u,\, \dP_{0}w)=\int_{\partial\Gamma_0} \overline{\partial_{\nu}u}w_1-\overline{u}w_2.
  \end{equation}
 We estimate each term in \eqref{vv35new} individually. Using
  $w_1\in D_0$ and the Cauchy--Schwarz inequality one obtains
  \begin{align}
    \begin{split}\label{vv36}
      \Big|\int_{\partial\Gamma_0} \overline{\partial_{\nu}u}w_1\Big|&=\Big|\sum_{e\in\cE_0}w_1(b_e)u'(b_e)-w_1(a_e)u'(a_e)\Big|=\Big|\sum_{e\in\cE_0}w_1(a_e)\int_{e}u''\Big|\\
      &\leq \sum_{e\in\cE_0}|w_1(a_e)|\sqrt{\ell_e} \|u''\|_{L^2(e)}.
    \end{split}
  \end{align}
  
  Similarly, using $w_2\in N_0$ and the Cauchy--Schwarz inequality we get
  \begin{align}
    \begin{split}\label{vv37}
      \Big|\int_{\partial\Gamma_0} \overline{u}w_2\Big|
      &=\Big|\sum_{e\in\cE_0}w_2(b_e)\overline{u}(b_e)
      +w_2(a_e)\overline{u}(a_e)\Big|
      =\Big|\sum_{e\in\cE_0}w_2(a_e)\int_{e}\overline{u}'\Big|\\
      &\leq \sum_{e\in\cE_0}|w_2(a_e)|\sqrt{\ell_e} \|u'\|_{L^2(e)}.
    \end{split}
  \end{align}
  
  Therefore, utilizing  \eqref{vv33}, \eqref{vv35new} -- \eqref{vv37} we arrive at
  \begin{align}
    \big|Z\big|
    \lesssim \sqrt{|\ell_0|}\, 
    \|\dP_0 w\|_{\dL^2(\partial\Gamma_+)} 
    \|u\|_{\hatt{H}^2(\Gamma(\ell))}
    &\lesssim \sqrt{|\ell_0|}\,\|\tr v\|_{\dL^2(\partial\Gamma_+)} 
      \|u\|_{\hatt{H}^2(\Gamma(\ell))} \\
    & \leq  \sqrt{|\ell_0|}\,\|\tr\|\, 
      \|v\|_{\hatt{H}^2(\Gamma(\wti\ell))}
      \|u\|_{\hatt{H}^2(\Gamma(\ell))}
  \end{align}
  Let us notice that 
  \begin{equation}\label{az6.1}
    \sup\limits_{\ell_+:\ 
      \frac{|\wti \ell|}{2}\leq |\ell_+|\leq |\wti\ell|}
    \|\tr^{\hatt\ell}\|_{\cB\big(\hatt{H}^2(\Gamma(\ell)),
      \dL^2(\partial\Gamma)\big)}
      =\cO\left(1\right)\  \text{as\ }
    \ell_0\rightarrow 0,
  \end{equation}
  and
  \begin{align*}
    \|u\|_{\hatt{H}^2(\Gamma(\ell))} 
    &\leq \|R(\cL,\ell,\overline{z})\|_{\cB\big(L^2(\Gamma(\ell)),
      \hatt{H}^2(\Gamma(\ell))\big)}
      \|f\|_{L^2(\Gamma(\wti\ell))} \\
    \|v\|_{\hatt{H}^2(\Gamma(\wti\ell))}
    &\leq \|R(\wti\cL,\hatt\ell,z)\|_{\cB\big(L^2(\Gamma(\wti\ell)),
      \hatt{H}^2(\Gamma(\ell))\big)}
      \|g\|_{L^2(\Gamma(\ell))}
  \end{align*}
  
  Combining these with \eqref{b8} we obtain \eqref{u22new}.
	
  {\bf Step 2.}  Let us denote $ \hatt\ell:= (\ell_+, 0)$,
  $\wti\ell=(\wti\ell_+,0)$ and let $\cJ_{\hatt\ell}:
  L^2(\Gamma(\wti\ell)) \to L^2(\Gamma(\hatt\ell))$ be defined as
  \begin{equation*}
    \left(\cJ_{\hatt\ell} f\right)(x) 
    = \sqrt{\frac{\wti{\ell}_e}{\ell_e}}
    f\left(\frac{x\wti{\ell}_e}{\ell_e}\right), 
    \quad x \in e \in \cE_+.
  \end{equation*}
  We remark that in this case, the operators $\cJ_{\hatt \ell}$ are unitary.
  We need to prove
  \begin{equation}
    \label{az10}
    \Big\|{\cJ_{\hatt\ell}\,}R(\wti\cL,\wti\ell,z)-R(\wti\cL,\hatt\ell,z){\cJ_{\hatt\ell}\,}\Big\|_{\cB\left(L^2(\Gamma(\wti\ell)), L^2(\Gamma(\hatt\ell))\right)}\underset{\ell_+\rightarrow \wti\ell_+}{=} o\left(1\right),
  \end{equation}
  where $R(\wti\cL,\hatt\ell,z)$ denotes the resolvent of
  $H(\wti\cL, \hatt\ell\,)$, the Laplace operator acting in
  $L^2(\Gamma(\hatt \ell))$ and associated with the Lagrangian plane
  $\wti\cL$ as in Theorem \ref{a30.1}.
	
  This case has been considered in \cite[Theorem 3.7]{BK12}. In
  particular, it is proved there that the operator valued function
  \begin{equation}
    \hatt\ell\mapsto{\cJ_{\hatt\ell}\,}R(\wti\cL,\hatt \ell,z){\cJ_{\hatt\ell}}^{-1},
  \end{equation} 
  is continuous. This together with the fact that $\cJ_{\hatt\ell}$
  is unitary implies \eqref{az10}.

  {\bf Step 3.} In this step we show how to combine the results from
  previous steps to derive \eqref{u8}. To this end we use \eqref{d3}
  and $\cJ_{\ell}=\bbJ_{\ell,\hatt\ell}\,\cJ_{\hatt\ell}$ to notice
  the following:
  \begin{multline}
    \cJ_{\ell}\,R(\wti\cL,\wti\ell,z)-R(\cL,\ell,z)\cJ_{\ell}=\bbJ_{\ell,\hatt\ell}\,\cJ_{\hatt\ell}\,R(\wti\cL,\wti\ell,z)-R(\cL,\ell,z)\bbJ_{\ell,\hatt\ell}\,\cJ_{\hatt\ell} \label{pr5}\\
    =\big(\bbJ_{\ell,\hatt\ell}\,R(\wti\cL,\ell,z)-R(\cL,\ell,z)\bbJ_{\ell,\hatt\ell}\,\big)\cJ_{\hatt\ell}
    +\bbJ_{\ell,\hatt\ell}\,\big( {\cJ_{\hatt\ell}\,}R(\wti\cL,\wti\ell,z)-R(\wti\cL,\ell,z)\,{\cJ_{\hatt\ell}\,}\big).
  \end{multline}
  For all $\ell, \hatt\ell$ one has
  \begin{equation}
    \|\cJ_{\hatt\ell}\, \|_{\cB(L^2(\Gamma(\wti\ell), \Gamma(\hatt \ell)))}=1,\ \|\bbJ_{\ell,\hatt\ell}\, \|_{\cB(L^2(\Gamma(\hatt\ell), \Gamma(\ell)))}=1.
  \end{equation}
  Moreover,
  \begin{equation}
    \ell=(\ell_+,\ell_0)\rightarrow\wti\ell= (\wti\ell_{+}, 0)\iff \ell_+\rightarrow\wti\ell_{+},\ \ell_0\rightarrow 0.
  \end{equation}
  Therefore, using \eqref{u8new} (uniformly in $\ell_+$ satisfying
  \eqref{az5}), \eqref{az10}, \eqref{pr5} and the triangle inequality,
  we obtain \eqref{u8}.
\end{proof}

Our next goal is to show that the generalized resolvent convergence of
the Schr\"odinger operators implies convergence of spectral
projections and thus convergence of spectra in the Hausdorff sense. In
case of non-negative operators this result was established in
\cite[Theorem 4.3.3]{P12}. In the present setting \cite[Theorem
4.3.3]{P12} is not directly applicable since the bottom of the
spectrum of $H(\cL,\ell)$ may tend to negative infinity as
$\ell\rightarrow\wti\ell$ (cf., \cite[Section
3.3]{KS06}). Nevertheless, the convergence of spectra still holds. We
carry out the proof following the standard line of arguments from
\cite{P12}, \cite[Theorem VIII.20, VIII.23, VIII.24]{RS1}.

\begin{proof}[Proof of Theorem~\ref{thm:spectral_convergence}]\label{pgspec}

  The convergence of spectra follows from Theorem~\ref{cc1} and
  \cite[Proposition 4.3.1]{P12}.
	First of all, by Theorem~\ref{y1}, conditions {\it(i)-(iii)} of
  Theorem~\ref{ss1} hold. Hence Theorem \ref{uu77} and Theorem
  \ref{u7} are applicable.
	
	Next, in order to simplify notation let us denote 
	\begin{align}
	\begin{split}\label{ee1}
	&\wti R_{\pm}:=R(\wti\cL,\wti\ell, \pm\bfi),\ \ R_{\pm}:=R(\cL,\ell, \pm\bfi),\\
	&\wti H:=H(\wti\cL,\wti\ell),\ \  H:=H(\cL,\ell).
	\end{split}
	\end{align}
	
	Let us prove the first assertion in \eqref{eq:projectors_convergence_AB}. Proceeding as in \cite[Theorem 4.2.9]{P12} and using \eqref{u8} we get
	\begin{align}
	\begin{split}\label{ee7}
	\Big\|\cJ_{\ell}\wti R^p_{\pm}-R^p_{\pm}\cJ_{\ell}\Big\|_{\cB\left(L^2(\Gamma(\wti\ell)), L^2(\Gamma(\ell))\right)}\underset{\ell\rightarrow\wti\ell}{=} o(1),\ p\in\bbN.
	\end{split}
	\end{align}
	Next, for arbitrary $p,q\in\bbN$ one has
	\begin{align}
	\cJ_{\ell}\wti R^p_{+}\wti R^q_{-}-R^p_{+}R^q_{-}\cJ_{\ell}&= \big(\cJ_{\ell}\wti R^p_{+}-R^p_{+}\cJ_{\ell}\big)\wti R^q_{-}\label{ee8}\\
	&\quad +R^p_{+}\big( \cJ_{\ell}\wti R^q_{-}-R^q_{-}\cJ_{\ell}\big).\label{ee9}
	\end{align}
	Let us notice that 
	\begin{equation}
	\|R^p_{+}\|_{\cB(L^2(\Gamma(\ell)))}\leq 1,\ \|\wti R^p_{+}\|_{\cB(L^2(\Gamma(\wti\ell)))}\leq 1.
	\end{equation}
	Therefore \eqref{ee7}--\eqref{ee9} yield 
	\begin{equation}\label{ee10}
	\Big\|\cJ_{\ell}\wti R^p_{+}\wti R^q_{-}-R^p_{+}R^q_{-}\cJ_{\ell}\Big\|_{\cB\left(L^2(\Gamma(\wti\ell)), L^2(\Gamma(\ell))\right)}\underset{\ell\rightarrow\wti\ell}{=} o(1),\ p,q\in\bbN.
	\end{equation}
	By the Stone--Weierstrass theorem polynomials in $(x+\bfi)^{-1}$ and $(x-\bfi)^{-1}$ are dense in $C(\overline{\bbR})$, the space of continuous functions for which the limits at both $+\infty$ and $-\infty$ exist and are equal. That is, given any $f\in C(\overline{\bbR})$ and arbitrary $\varepsilon>0$ there exits a polynomial $P(u,v)$ such that
	\begin{equation}\label{ee30}
	\text{ess sup}_{x\in\bbR}|f(x)-P((x+\bfi)^{-1}, (x-\bfi)^{-1})|<\varepsilon.
	\end{equation} 
	Combining \eqref{ee10} and \eqref{ee30} we arrive at 
	\begin{align}
	\begin{split}\label{ee11}
	\Big\|\cJ_{\ell}f(\wti H)-f( H)\cJ_{\ell}\Big\|_{\cB\left(L^2(\Gamma(\wti\ell)), L^2(\Gamma(\ell))\right)}\underset{\ell\rightarrow\wti\ell}{=} o(1),
	\end{split}
	\end{align}
	for all $f\in C(\overline{\bbR})$. As in the case of positive operators, \eqref{ee11} gives rise to a similar identity for the spectral projections corresponding to bounded open sets. Namely, in the present context the analogue of \cite[Corollary 4.2.12]{P12} reads as 
	\begin{align}
	\begin{split}\label{ee31}
	\Big\|\cJ_{\ell}\chi_{(a,b)}(\wti H)-\chi_{(a,b)}(H)\cJ_{\ell}\Big\|_{\cB\left(L^2(\Gamma(\wti\ell)), L^2(\Gamma(\ell))\right)}\underset{\ell\rightarrow\wti\ell}{=} o(1),
	\end{split}
	\end{align}
	where $-\infty<a<b<\infty$ and $a,b\not\in\spec(\wti H)$. In order to show \eqref{ee31}, let us pick any $\psi\in C(\overline{\bbR})$ satisfying 
	\begin{equation}
          0\leq \psi\leq 1,\ \supp(\psi)\subset\bbR\setminus\{a,b\}
          \quad\mbox{and}\quad\psi(x)=1\mbox{ whenever } x\in\spec(\wti H).
	\end{equation}
	Then
	\begin{multline}
          \Big\|\cJ_{\ell}\chi_{(a,b)}(\wti H)-\chi_{(a,b)}(H)\cJ_{\ell}\Big\|
          \leq\Big\|\cJ_{\ell}\psi(\wti H)\chi_{(a,b)}(\wti
          H)-\psi(H)\chi_{(a,b)}(H)\cJ_{\ell}\Big\| \\
          +\Big\|\cJ_{\ell}(1-\psi)(\wti H)\chi_{(a,b)}(\wti H)-(1-\psi)(H)\chi_{(a,b)}(H)\cJ_{\ell}\Big\|,\label{ee41}
	\end{multline}
	where the norms are taken in ${\cB\left(L^2(\Gamma(\wti\ell)), L^2(\Gamma(\ell))\right)}$. Since $\psi\chi_{(a,b)}\in C(\overline{\bbR})$, the expression in \eqref{ee41} is $o(1)$ as $\ell\rightarrow\wti\ell$. Using $(1-\psi)(\wti H)=0$, we rewrite and estimate \eqref{ee41} as follows
	\begin{equation}
          \|(1-\psi)(H)\chi_{(a,b)}(H)\cJ_{\ell}\Big\| 
          \leq \Big\|(1-\psi)(H)\cJ_{\ell}\Big\|
          \leq \Big\|\cJ_{\ell}(1-\psi)(\wti
          H)-(1-\psi)(H)\cJ_{\ell}\Big\|.
          \label{ee43}
	\end{equation}
	Since  $1-\psi \in C(\overline{\bbR})$, the expression in \eqref{ee43} is $o(1)$ as $\ell\rightarrow\wti\ell$. Hence, \eqref{ee31} and the first part of  \eqref{eq:projectors_convergence_AB} hold as asserted. Analogously, the second part of  \eqref{eq:projectors_convergence_AB}  can be derived from the second part of \eqref{eq:resolvent_convergence_AB}. 
\end{proof}

\begin{proof}[Proof of Theorem \ref{scale}]\lb{pgscale}
  Due to Theorem~\ref{thm:spectral_convergence} it is enough to show
  that \eqref{c9} implies Condition
  \ref{newhyp}. Seeking a contradiction we assume that Condition
  \ref{newhyp} is not fulfilled and will show that \eqref{c9} does not
  hold. In fact we will prove a slightly stronger statement,
  \begin{equation}
    \lb{gh21}
    \dim(\ker(H(\cL,\ell))) 
    > \dim(\ker(H(\wti\cL,\wti\ell))),
    \quad \ell\in\bbR^{|\cE|}_{>0}.
  \end{equation}
  In particular, the multiplicity of zero eigenvalues of the limiting
  and the approximating operators do not match.

  Our first objective is to prove that any
  $\varphi\in \ker(H(\wti\cL,\wti\ell))$ is constant on each edge. By
  Proposition \ref{gh1} there exist subspaces $\cL_D, \cL_N$ such that
  \begin{equation*}
    \cL=\{(\phi_1,\phi_2)\in\,\dL^2(\partial \Gamma): \phi_1\in\cL_D,  \phi_2\in\cL_N \}.
  \end{equation*}
  By Theorem \ref{a30.1}, one has 
  \begin{equation}
    \wti\cL := \{ \dP_+(\phi_1,\phi_2): 
    \phi_1\in \cL_D\cap D_0, \phi_2\in \cL_N\cap N_0\}.
  \end{equation}
  Then by Proposition \ref{gh1} the vertex conditions of
  $H(\wti\cL,\wti\ell)$ are scale invariant. From~\eqref{aa14} one has
  \begin{equation}
    0 = \langle\varphi, 
    H(\wti\cL,\wti\ell)\varphi\rangle_{L^2(\Gamma(\wti\ell))}
    = \|\varphi'\|_{L^2(\Gamma(\wti\ell))}.
  \end{equation}
  Thus $\varphi$ is constant on each edge, in particular $\varphi(a_e)
  = \varphi(b_e)$ for every $e\in\cE_+$.

  Next, for each $\varphi\in\ker(H(\wti\cL,\wti\ell))$ we construct
  $f_\varphi\in\ker(H(\cL,\ell))$ as follows. Since
  $\tr\varphi\in \wti\cL$, by Theorem \ref{a30.1} there exists
  \begin{equation}
    (\phi_1,\phi_2)\in \cL\cap (D_0\oplus N_0)
  \end{equation} 
  such that $\tr\varphi = \dP_+(\phi_1,\phi_2)$.  Note that since
  $\varphi$ was constant on edges from $\cE_+$ and $\phi_1 \in D_0$,
  \begin{equation}
    \label{eq:phi_equal_everywhere}
    \phi_1(a_e) = \phi_1(b_e)\quad\mbox{on every edge }e.
  \end{equation}
  Let us define a function $f_{\varphi}$, constant on each edge, by
  the formula
  \begin{equation}
    f_{\varphi} := \sum_{e\in\cE}\phi_1(a_e)\chi_e,
  \end{equation}
  We claim that $\tr f_{\varphi}\in\cL$.  By construction and
  property~\eqref{eq:phi_equal_everywhere},
  $\gamma_D f_\varphi = \phi_1$.  Since $f_\varphi$ is constant on
  edges, $\gamma_N f_\varphi = 0$.  Finally, by Proposition~\ref{gh1}
  $(\phi_1,\phi_2)\in\cL$ implies $(\phi_1,0)\in\cL$.  Therefore,
  $f_{\varphi}\in \ker(H(\cL,\ell))$ and
  $f_{\varphi}|_{\Gamma_+}=\varphi$.

  We have now produced a function $f_\varphi\in\ker(H(\cL,\ell))$ for
  every $\varphi\in\ker(H(\wti\cL,\wti\ell))$.  It it easy to see that
  $f_\varphi$ are linearly independent if the corresponding $\varphi$
  are.  Furthermore, no non-trivial linear combination of $f_\varphi$
  can be zero on $\Gamma_+$.

  Let us now utilize Lemma \ref{y21} to produce a nonzero
  $f\in\ker(H(\cL,\ell))$ such that $f|_{\Gamma_+}=0$.  It is clearly
  linearly independent of all $f_\varphi$, leading to
  \begin{equation}
    \dim(\ker(H(\wti\cL,\wti\ell))) 
    < \dim(\text{span} \{f_{\varphi}, f: \varphi\in
    \ker(H(\wti\cL,\wti\ell))\})
    \leq \dim(\ker(H(\cL,\ell)))
  \end{equation}
  as required.
\end{proof}

\end{document}